\documentclass[preprint,12pt]{elsarticle}
\usepackage{pdfsync}
\usepackage{epsfig}
\usepackage{amsmath}
\usepackage{amssymb}
\usepackage{graphicx}
\usepackage[latin1]{inputenc}
\usepackage{amsthm}
\usepackage[english]{babel}
{\catcode `\@=11 \global\let\AddToReset=\@addtoreset}
\AddToReset{equation}{section}

\newtheorem{theorem}{Theorem}[section]
\newtheorem{lemma}{\bf Lemma}[section]

\newtheorem{@definition}{\sc Definition}[section]

\newtheorem{@remark}{\sc Remark}[section]

\newtheorem{@example}{\sc Example}[section]

\newcommand{\beqn}{\begin{displaymath}}
\newcommand{\eeqn}{\end{displaymath}}
\newcommand{\beq}{\begin{equation}}  % numbered (single equation)
\newcommand{\eeq}{\end{equation}}

\def\mathsf{\bf}
\def\N{\mathbb{N}}

\def\R{\mathbb{R}}
\def\Z{\mathbb{Z}}

\def\E{\mathrm E}
\def\P{\mathrm P}

\def\text{\mbox}

\def\1{{\bf 1}}

\newcommand{\Var}{\mbox{Var}}
\newcommand{\Cov}{\mbox{Cov}}

\DeclareMathOperator*{\argmin}{arg\,min}

\def\limiteloin{\renewcommand{\arraystretch}{0.5}
\begin{array}[t]{c}
\stackrel{{\cal D}}{\longrightarrow} \\
{\scriptstyle n \rightarrow\infty}
\end{array}\renewcommand{\arraystretch}{1}}

\def\limiteprobamn{\renewcommand{\arraystretch}{0.5}
\begin{array}[t]{c}
\stackrel{{\cal P}}{\longrightarrow} \\
{\scriptstyle m, \, n/m \rightarrow\infty}
\end{array}\renewcommand{\arraystretch}{1}} 

\def\limiteproban{\renewcommand{\arraystretch}{0.5}
\begin{array}[t]{c}
\stackrel{{\cal P}}{\longrightarrow} \\
{\scriptstyle n \rightarrow\infty}
\end{array}\renewcommand{\arraystretch}{1}}

\def\limiteprobanm{\renewcommand{\arraystretch}{0.5}
\begin{array}[t]{c}
\stackrel{{\cal P}}{\longrightarrow} \\
{\scriptstyle n,\, m ,\, n/m \rightarrow\infty}
\end{array}\renewcommand{\arraystretch}{1}}

\def\limiten{\renewcommand{\arraystretch}{0.5}
\begin{array}[t]{c}
\stackrel{}{\longrightarrow} \\
{\scriptstyle n\rightarrow\infty}
\end{array}\renewcommand{\arraystretch}{1}}

\def\limitemn{\renewcommand{\arraystretch}{0.5}
\begin{array}[t]{c}
\stackrel{}{\longrightarrow} \\
{\scriptstyle n,\,m,\,n/m\rightarrow\infty}
\end{array}\renewcommand{\arraystretch}{1}}

\def\limitet{\renewcommand{\arraystretch}{0.5}
\begin{array}[t]{c}
\stackrel{}{\longrightarrow} \\
{\scriptstyle t\rightarrow\infty}
\end{array}\renewcommand{\arraystretch}{1}}

\def\limitet0{\renewcommand{\arraystretch}{0.5}
\begin{array}[t]{c}
\stackrel{}{\longrightarrow} \\
{\scriptstyle t\rightarrow 0}
\end{array}\renewcommand{\arraystretch}{1}}

\newtheorem{rem}{Remark}
\newtheorem{cor}{Corollary}

\def\Cov{\mathrm{Cov}}

\journal{Stochastic Processes and Applications}

\begin{document}
\begin{frontmatter}

\title{Data-driven semi-parametric detection of multiple changes in long-range dependent processes}

\author{Jean-Marc Bardet and Abdellatif Guenaizi} 
\address{{\tt bardet@univ-paris1.fr}, {\tt Abdellatif.Guenaizi@malix.univ-paris1.fr} \\
~\\ 
SAMM, Universit\'e Panth\'eon-Sorbonne (Paris I), 90 rue de Tolbiac, 75013 Paris, FRANCE}

\begin{abstract}
This paper is devoted to the offline multiple changes detection for long-range dependence processes. 
The observations are supposed to satisfy a semi-parametric long-range dependence assumption with distinct memory parameters on each stage. A penalized local Whittle contrast is considered for estimating all the parameters, notably the number of changes. The consistency as well as convergence rates are obtained.  Monte-Carlo experiments exhibit the accuracy of the estimators. They also show that the estimation of the number of breaks is improved  by using a data-driven slope heuristic procedure of choice of the penalization parameter.
\end{abstract}
\end{frontmatter}

\section{Introduction}
There exists now a very large literature devoted to long-range dependent processes. The most commonly used definition of long-range dependency requires a second order stationary process $X=(X_n)_{n\in \Z}$ with spectral density $f$ such as:
\begin{equation}\label{density}
f(\lambda)=|\lambda| ^{-2d} \, L\big ( |\lambda|\big )\quad \mbox{for any $\lambda \in [-\pi, \pi]$},
\end{equation}
where $L$ is a positive slow varying function, satisfying for any $c>0$, $\lim_{h \to 0} \frac {L(c \, |h|)}{L(|h|)}=1$, typically $L$ is a function with a positive limit or a logarithm. \\
From an observed trajectory $(X_1,\ldots,X_n)$ of a long-range dependent process, the estimation of the parameter $d$ is an interesting statistical question. The case of a parametric estimator for which the explicit expression of the spectral density $f$ is known, was successively solved in many cases using maximum likelihood estimators (see for instance Dahlhaus, 1989) or Whittle estimators (see for instance Fox and Taqqu, 1987, Giraitis and Surgailis, 1990, or Giraitis and Taqqu, 1999).  \\
However, with numerical applications in view, knowing the explicit form of the spectral density is not a realistic framework. A semi-parametric estimation of $d$ where only the behaviour \eqref{density} is assumed should be preferred. Thus, numerous semi-parametric estimators of $d$ were defined and studied, the main ones being the log-periodogram (see Geweke and Porter Hudak, 1987, or Robinson, 1995a), the wavelet based (see Bardet {\it et al.}, 2000) and the local Whittle estimators (see Robinson 1995b). \\
This last one is a version of the Whittle estimator for which only asymptotically small frequencies are considered. It provides certainly the best trade-off between computation time and accuracy of the estimation (see for instance Bardet {\it et al.}, 2003b). Its asymptotic normality was extended for numerous kinds of long-memory processes (see Dalla {\it et al.}, 2006) and also non-stationary processes (see Abadir {\it et al.}, 2006). However there is still not satisfactory adaptive method of choice of the bandwidth parameter even if several interesting attempts have been developed (see for instance Henry and Robinson, 1998, or Henry, 2007). Hence, the usual choice valid for FARIMA or Fractional Gaussian noise is commonly chosen. \\
~\\
In this paper we consider the classical framework of offline multiple change detection. It consists on the observed trajectory $(X_1,\ldots,X_n)$ of a process $X$ whose trajectory is partitioned into $K^*+1$ subtrajectories on which it is a linear long memory process whose long memory parameters are distinct from one area to another (see a more precise definition in \eqref{deflin}). Thus there is dependence between two subtrajectories since all the different linear processes are constructed from the same white noise. The aim of this paper is to present a method for estimating from $(X_1,\ldots,X_n)$ the number $K^*$ of abrupt changes, the $K^*$ change-times $(t_1^*, \ldots,t_{K^*}^*)$ and the $K^*+1$ different long-memory parameters $(d_1^*,\ldots,d_{K^*+1}^*)$, which are unknown. \\
The framework of ``offline'' multiple changes we chose, has to be distinguished from that of the ``online'' one, for which a monitoring procedure is adopted and test of detection of change is successively applied (such as CUSUM procedure). The book of Basseville and Nikiforov (1993) is a good reference for an introduction on both online and offline methods. There exist several methods for building a sequential detector of long-range memory, see for instance Giraitis {\it et al.} (2001), Kokoszka and Leipus (2003) or Lavancier {\it et al.} (2013).  \\
For our offline framework, following the previous purposes, we chose to build a penalized contrast based on a sum successive local Whittle contrasts and to minimize it. The principle of this method, minimizing a penalized contrast, provides very convincing results in many frameworks: in case of mean changes with least squares contrast (see Bai, 1998), in case of linear models changes with least squares contrast (see Bai and Perron, 1998, generalized by Lavielle, 1999, and Lavielle and Moulines, 2000) or least absolute deviations (see Bai, 1998), in case of spectral densities changes with usual Whittle contrasts (see Lavielle and Ludena, 2000), in case of time series changes with quasi-maximum likelihood (see Bardet {\it et al.}, 2012),... Clearly, the remarkable paper of Lavielle and Ludena (2000) was the model of this article except that we used a semi-parametric version of their Whittle contrast with the local Whittle contrast, and this engenders additional difficulties... \\
Restricting our paper to long-memory linear processes, we obtained several asymptotic results. First the consistency of the estimator has been established under assumptions on the second order term of the expansion of the spectral density close to $0$. A convergence rate of the change times estimators is also provided, but we are not able to reach the usual ${\cal O}_{\P}(1)$ converge rate, which is obtained for instance in the parametric case (see Lavielle and Ludena, 2000).\\
Monte-Carlo experiments illustrate the consistency of the estimators. When the number of changes is known, the theoretical results concerning the consistencies of the estimator are satisfying and $n=5000$ provides very convincing results while they are still mediocre for $n=2000$ and bad for $n=500$. This is not surprising since we considered a semi-parametric statistical framework. When the number of changes is unknown, although we chose an asymptotically consistent choice of penalization sequence, the consistency is not satisfying even for large sample such as $n=5000$. The accuracy of the number of changes estimator is extremely dependent on the precise choice of the penalization sequence, even if this choice should not be important asymptotically. Then we chose to use a data-driven procedure for computing ``optimal'' penalty, the so-called ``Slope Heuristic'' procedure defined in Arlot and Massart (2009). It provides more accurate results than with a fixed penalization sequence and it leads to very convincing results when $n=5000$. \\
~\\
The following Section \ref{defpte} is devoted to define the framework and the estimator. Its asymptotic properties are studied in Section \ref{asympto}. The concrete
estimation procedure and numerical applications are presented in
Section \ref{simu}. Finally, Section \ref{proofs} contains the main
proofs.
\section{Definitions and assumptions}\label{defpte}
\subsection{The multiple changes framework}
We consider in the sequel the case of multiple change long-range dependent linear processes. First we define a class $L(d,\beta,c)$ of real sequences, where $d\in [0,1/2)$, $\beta \in (0,2]$ and $c>0$:\\
~\\
{\bf Class $L(d,\beta,c)$:~} {\em A sequence $(a_i)_{i\in \N} \in \R^\N$ belongs to the class $L(d,\beta,c)$ if 
\begin{itemize}
\item  $|a_n|=c\, n^{d-1}+ O\big (n^{d-1-\beta} \big) $ when $n\to \infty$;
\item $\frac {\partial}{\partial \lambda} \alpha(\lambda)=O\big (\big |\lambda^{-1} \, \alpha(\lambda)\big | \big )$ when $\lambda\to 0^+$ with $\alpha(\lambda)=\sum_{j=0}^\infty a_j e^{ij\lambda}$.\\
\end{itemize}}
\noindent Note that the class $L(d,\beta,c)$ is included in $\ell^2(\R)$, the Hilbert space of square summable sequences.  \\
Now, for $(a_i)_{i\in \N} \in \R^n$ a sequence of the class $L(d,\beta,c)$, it is possible to define a second order linear long-range dependent process. Indeed, with $(\varepsilon_t)_{t\in \Z}$  a sequence of independent and identically distributed random variables (iidrv) with zero mean and unit variance, we can define 
$Y=(Y_k)_{ k \in \Z}$ such as
$$
Y_k=\sum_{j=0}^\infty a_j \, \varepsilon_{k-j} \qquad\mbox{for $k \in \Z$}.
$$
Note that $Y$ is a zero mean stationary process, with autocovariance $r(k)=\E (Y_0Y_k)$ satisfying 
\begin{equation}\label{covr}
r(n)=c^2 B(1-2d,d)\, n^{2d-1}+ O\big ( n^{2d-2-\beta}\big )\quad \mbox{when $u\to \infty$},
\end{equation}
with $B(u,v)$ the usual Beta function (see for instance Inoue, 1997). It is also possible to define the spectral density $f$ of $Y$ in $[-\pi,0)\cup(0, \pi]$ and it satisfies for $d \in (0,1/2)$
\begin{eqnarray}\label{Spectral}
f(\lambda) =  \frac {c^2}\pi \, B(1-2d,d) \, \Gamma(2d)\, \sin \Big (\frac \pi 2 -\pi d \Big )\, \big |\lambda \big |^{-2d}+O\big (|\lambda|^{-2d+\beta}\big )\quad \mbox{when $\lambda \to 0$},
\end{eqnarray} 
using the Tauberian Theorem in Zygmund (1968) and with $\Gamma(u)$ the usual Gamma function. By the way, we can also write that there exists $c'>0$ such as
\begin{eqnarray}\label{AssumptionS}
f(\lambda) = c'\, \big |\lambda \big |^{-2d}+O\big (|\lambda|^{-2d+\beta}\big )\quad \mbox{when $\lambda \to 0$},
\end{eqnarray}
that is the classical assumption required for instance in Robinson (1995b).\\
~\\
Using these definitions, we are going to give the following assumption satisfied by the trajectory $(X_1,\ldots,X_n)$ of the process $X$ from we study the changes: \\
~\\
{\bf Assumption $A$:} {\it Let $(\varepsilon_t)_{t\in \Z}$ be a sequence of iidrv with zero mean and unit variance. Denote also:
\begin{itemize}
\item $K^* \in \{0,\ldots,n-1\}$, $\tau^*_0=0<\tau^*_1<\cdots<\tau^*_{K^*}<1=\tau^*_{K^*+1}$; 
\item $(d_i^*)_{1\leq i \leq K^*+1} \in [0,1/2)^{K^*+1}$, $(c_i^*)_{1\leq i \leq K^*+1} \in (0,\infty)^{K^*+1}$ and $(\beta_i^*)_{1\leq i \leq K^*+1} \in (0,2] ^{K^*+1}$
\item $K^*+1$ sequences $(a^{(i)}_t)$ such as $ (a^{(i)}_t)_{t\in \N}$ belongs to the class $L(d_i^*,\beta_i^*,c_i^*)$ for all $i=1,\cdots,K^*+1$.
\end{itemize}
Define the process $X=(X_t)_{1\leq t \leq n}$ such as
\begin{enumerate} 
\item for $i=1,\cdots,K^*+1$,  
\begin{eqnarray}\label{deflin}
X_t=\sum_{j=0}^\infty a^{(i)}_j \, \varepsilon _{t-j}\quad \mbox{when $[n\tau_{i-1}^*]+1 \leq t \leq [n\tau_{i}^*]$.}
\end{eqnarray} 
\item For $i=1,\cdots,K^*$, $d^*_{i+1}-d^*_i\neq 0$ and denote
\begin{eqnarray}\label{Deltad}
\Delta_d=\max_{1\leq i \leq K^*} \big |d^*_{i+1}-d^*_i \big |>0.
\end{eqnarray}
\end{enumerate}}
The first condition \eqref{deflin} is relative to the behavior $(X_t)$ in each stage: it is a stationary linear long-range process with a spectral density satisfying \eqref{AssumptionS} (where $d=d_i^*$). Moreover there also exists a dependence  for $(X_t)$ from one stage to another one (see for instance the proof of Lemma \ref{lem} where the covariance between two subtrajectories of $X$ is computed in \eqref{tt'}), which makes the model much more realistic than if the independence of successive regimes had been assumed. The second condition \eqref{Deltad} is  the key condition insuring that the framework is the one of multiple long-range dependence change. 
\subsection{Definition of the estimator}
First we will add other notation: \\
~\\
For $X$ satisfying Assumption A, denote: 
\begin{itemize}
 \item $t_i^*=[n\tau_i^*]$, $T^*_i=\big \{t^*_{i-1}+1,t^*_{i-1}+2,\cdots,t^*_{i} \big \}$ and $n^*_i=t^*_i-t^*_{i-1}$ for $i=1,\ldots,K^*+1$.
\end{itemize} 
More generally, for $K \in \{0,1,\cdots,n-1\}$ and $t_0=1<t_1<\cdots <t_K<t_{K+1}=n$, 
\begin{itemize}
\item  denote $T_i=\big \{t_{i-1}+1,t_{i-1}+2,\cdots,t_{i} \big \}$ and $n_i=t_i-t_{i-1}$ for $i=1,\ldots,K+1$. 
\item denote $T_{ij}=\big \{t_{i-1}+1,t_{i-1}+2,\cdots,t_{i} \big \}\cap \big \{t^*_{j-1}+1,t^*_{j-1}+2,\cdots,t^*_{j} \big \} $ and $n_{ij}=\# \{ T_{ij} \}$ for $i=1,\ldots,K+1$ and $j=1,\ldots,K^*+1$.
\end{itemize}
For $K \in \{0,\ldots,n\}$, we will also use the following multidimensional notation: 
\begin{itemize}
\item ${\bf d}=(d_1,\cdots,d_{K+1})$ and ${\bf d^*}=(d_1^*,\cdots,d^*_{K^*+1})$, 
\item ${\bf t}=(t_1,\cdots,t_{K})$, $ {\bf t^*}=(t_1^*,\cdots,t^*_{K^*})$ and $\boldsymbol{\tau^*}= (\tau^*_1,\ldots,\tau^*_{K^*})$.
\end{itemize}
\noindent  From Assumption A, denote by $I_{T}$ the periodogram of $X$ on the set $T$ where $T \subset \{1,\ldots,n\}$, and denote $|T|=\#\{T \}$:
\begin{eqnarray} \label{defwI}
%w_{s,s'}(\lambda)=\sum_{k=s+1}^{t}X_{k}e^{-i\,k \, \lambda}\qquad \mbox{and}\qquad 
I_T(\lambda)=\frac{1}{2\pi \,|T|}\, \Big|\sum_{k \in T}X_{k}e^{-i\,k \, \lambda}\Big|^{2}.
\end{eqnarray}
~\\
\noindent Using the seminal papers of Kunsh (1987), Robinson (1995b) and Robinson and Henry (2003), we define a local Whittle estimator of $d$. For this, define for $T \subset \{ 1,\cdots,n\}$, $d \in \R$ and $m\in \{1, \cdots,n\}$,
\begin{eqnarray} \label{defW}
&& W_{n}(T,d,m)= \log \big (S_n(T,d,m)\big )-\frac  {2\, d} m \, \sum_{k=1}^{m} \log(k/m) \qquad \\
&& \hspace{1cm} \mbox{with} \qquad S_n(T,d,m)=\frac 1 m \, \sum_{j=1}^{m}\big (\frac j m \big )^{2d}I_{T}(\lambda_{j}^{(n)})  \quad \mbox{and} \quad \lambda^{(n)}_k=2\pi \,\frac  k {n}. \label{S0}
\end{eqnarray}
The local Whittle objective function  $d \to W_n(T,d,m)$ can be minimized for estimating $d$ on the set $T$ providing the local Whittle estimator $\widehat d=\argmin_{d \in [0,0.5)} W_n(T,d,m)$ on $T$. 
\begin{rem}
Note that we use Fourier frequencies $\lambda^{(n)}_k=2\pi \,\frac  k {n}$ in the definition of $W_{n}(T,d,m)$, while its common definition (see for instance Robinson, 1995b) consider the Fourier frequencies  $\lambda_k=2\pi \,\frac  k {|T|}$. The explanation of this choice stems from the fact that in the definition of the following contrast $L_n(K,{\bf t},{\bf d},m)$ on the whole trajectory $(X_1,\ldots,X_n)$ we will sum the local contrasts $W_{n}(T_k,d_k,m)$. This choice is required for allowing some simplifications in the proofs. But, as we assume that $|T^*_i|=n_i^*\sim (\tau_{i}-\tau_{i-1}^*)n$, we asymptotically use almost the usual frequencies.
\end{rem}
\noindent Under Assumption A, we expect to estimate the distinct $d^*_i$ on the different stages $\{t_i^*+1,\ldots,t_{i+1}^*\}$ by using several local Whittle contrasts. In addition we will obtaining a $M$-estimator for estimating $d_i^*$ but also $t_i^*$ and even $K^*$. Hence, for $m\in \{1, \ldots,n\}$, we consider now a penalized local Whittle contrast defined by:
\begin{eqnarray} \label{defJ}
J_n(K,{\bf t},{\bf d},m)= \frac 1 n \, \sum_{k=1}^{K+1} n_k\, W_{n}(T_k,d_k,m)+K \, z_n,
\end{eqnarray}
where $K \in \N$ is a number of changes, ${\bf d}\in [0,0.5)^{K+1}$, ${\bf t}\in {\cal T}_K(0)$ and $(z_n)$ is a sequence of positive real numbers that will be specified in the sequel.\\
This contrast is therefore a sum of local Whittle objective functions on the $K+1$ different stages $T_k$, $k=1,\ldots, K+1$, and a penalty term that is a linear function of the number of changes (and therefore of the number of estimated parameters). Then, with $K_{\max} \in \N^*$ a chosen integer number, we define:
\begin{eqnarray}\label{Esti2}
(\widehat K_n , \widehat {\bf t},\, \widehat {\bf d})=\argmin_{K \in \{0,\ldots,K_{\max}\},~ {\bf d}\in [0,0.5)^{K+1},~ {\bf t}\in {\cal T}_K(0)} J_n(K,{\bf t},{\bf d},m),
\end{eqnarray}
with  $\widehat {\bf d}=(\widehat d_1,\cdots,\widehat d_{\widehat K_n +1})\quad \mbox{and}\quad \widehat {\bf t}=(\widehat t_1,\cdots,\widehat t_{\widehat K_n })$, and where for $a\geq 0$, 
\begin{equation}\label{Ta}
{\cal T}_K(a)=\Big \{ (t_1,\ldots,t_{K}) \in \{2,\ldots,n-1\}^{K},~t_{i+1}>t_i~\mbox{and}~|t_i-t_i^*|\geq a\quad \mbox{for all}~i=1,\ldots,K \Big \}. 
\end{equation}

\section{Asymptotic behaviors of the estimators} \label{asympto}
\subsection{Case of a known number of changes}
We study first the case of a known number $K^*$ of changes. In such a framework, let us define two particular cases of the minimization of the function $J_n$. First denote $\widetilde {\bf t}=(\widetilde t_1,\cdots,\widetilde t_{K^* })$ and $\widetilde {\bf d}=(\widetilde d_1,\cdots,\widetilde d_{K^* +1})$ obtained when the number of changes is known and 
$\widehat {\bf d^*}=(\widehat d^*_i)_{1\leq i \leq K^*+1}$ obtained when the number of changes and the change dates are known. They are defined by:
\begin{eqnarray}\label{Esti*}
(\widetilde {\bf t},\, \widetilde {\bf d})=\argmin_{{\bf d}\in [0,0.5)^{K^*+1},~ {\bf t}\in {\cal T}_{K^*}(0)} J_n(K^*,{\bf t},{\bf d},m)\quad \mbox{and}\quad \widehat {\bf d^*}=\argmin_{{\bf d}\in [0,0.5)^{K^*+1},} J_n(K^*,{\bf t^*},{\bf d},m).
\end{eqnarray} 
Then, we can prove:
\begin{theorem}\label{theo1}
For $X$ satisfying Assumption A, with $\widetilde {\boldsymbol{\tau}} = (\widetilde \tau_1,\ldots,\widetilde \tau_{K^*})$  where $
\widetilde \tau_i= \frac {\widetilde t_i } n$ for $i=1,\cdots,K^*$, and if $m=o(n)$,
$$
( \widetilde {\boldsymbol{ \tau}}, \widetilde {\bf d}) \limiteproban (\boldsymbol{\tau}^*, {\bf d}^* ).
$$
\end{theorem}
\noindent This first theorem,  whose proof as well as all other proofs can be found in Section \ref{proofs}, can be improved for specifying the rate of convergence of the estimators:
\begin{theorem}\label{theo2}
For $X$ satisfying Assumption A, if $m =o\big ( n^{2\underline \beta^*/(1+2\underline \beta^*)} \big)$ where $\underline \beta^*=\min_{1\leq i \leq K^*+1} \beta_i^*$, then for any $\delta>0$,
\begin{equation} \label{conv1}
\lim_{\delta \to \infty} \quad\lim_{n\to \infty}\P \Big (\frac {\sqrt m} {n }  \, \big \| \widetilde {\bf t} - {\bf t^*}\big \| \geq \delta \Big ) = 0.
\end{equation}
\end{theorem}
\noindent This result provides a bound of the ``best'' convergence rate of $ \widetilde {\bf t}$ which is minimized by $ n^{(1+\underline \beta^*)/(1+2\underline \beta^*)}$, {\it i.e.} the ``best'' convergence rate for  $ \widetilde {\boldsymbol{\tau}}$ is minimized by $ n^{-\underline \beta^*/(1+2\underline \beta^*)}$. 
\begin{rem} \label{rem1}
This rate of convergence could be compared to the result obtained in the parametric framework of Lavielle and Ludena (2000) where the respective convergence rates (in probability) of $ \widetilde {\bf t}$ and $ \widetilde {\boldsymbol{ \tau}}$  are $1$ and $n^{-1}$. This is the price to pay for going from the parametric to the semi-parametric framework. But also the price to pay to the definition of local Whittle estimator which does not allow some simplifications as in the proof of Theorem 3.4 of Lavielle and Ludena (2000, p. 860). Indeed the random term of their classical used definition of Whittle contrast is $\int_{\pi}^\pi I_T(\lambda)/f(\lambda)d\lambda$  while our random term is $ \log\big ( \frac 1 m \, \sum_{j=1}^m (j/m)^{2d}  I_T(\lambda_j^{(n)})\big ) $: the logarithm term does not make possible  their simplifications. 
\end{rem}
\noindent Another consequence of this result is that there is asymptotically a small lose on the convergence rates of the long memerory parameter local Whittle estimators $\widetilde d_i$ when the change dates are estimated instead of being known. More formally, using the results of Robinson (1995b) improved by Dalla {\it et al.} (2006), we know that under conditions of Theorem \ref{theo2}, $\widehat d_i^*$ satisfies 
\begin{equation} \label{conv2}
\sqrt { m} \big (\widehat d_i^* -d_i^* \big ) \limiteloin {\cal N} \big (0 \, , \,\frac 1 4  \big ).
\end{equation}
Unfortunately, the rate of convergence obtained for $ \widetilde t_i$ in Theorem \ref{theo2} does not allow to keep this limit theorem when $\widehat d_i^*$ is replaced by $\widetilde d_i$. We rather obtain:
\begin{theorem}\label{theo2bis}
Under the assumptions of Theorem \ref{theo2}, for any $i=1,\ldots,K^*+1$,
\begin{equation} \label{conv3}
\sqrt {  m }\, \big |\widetilde d_i -d_i^* \big |= O_P(1).
\end{equation}
\end{theorem}
\subsection{Case of an unknown number of changes}
Here we consider the case where $K^*$ is unknown. For estimating $K^*$, the penalty term of penalized local Whittle contrast $J_n$ is now essantial. Indeed, we obtain:
\begin{theorem}\label{theo3}
Under the assumptions and notations of Theorem \ref{theo1}, if $K_{\max}\geq K^*$, with $m =o\big ( n^{2\underline \beta^*/(1+2\underline \beta^*)} \big)$ where $\underline \beta^*=\min_{1\leq i \leq K^*+1} \beta_i^*$ $z_n$ and $\max \big ( z_n \, , \,  \frac {1} {z_n \, \sqrt {  m }} \big ) \limiten 0$, using $(\widehat K, \,\widehat {\bf t},\, \widehat {\bf d})$ defined in (\ref{Esti2}), then
$$
(\widehat K, \widehat {\boldsymbol{\tau}}, \widehat {\bf d}) \limiteproban (K^*, \boldsymbol{\tau}^*, {\bf d}^* ).
$$
\end{theorem}
\noindent Note that the conditions we obtained on $m$ and $z_n$  imply that $n^{-\underline \beta^*/(1+2\underline \beta^*)} =o(z_n)$, depending on $\underline \beta^*$ that is generally unknown. However, the choice $z_n=n^{-1/2}$ is a possible choice solving this problem. The provided proof does not allow to establish the consistency of a typical BIC criterion, which should be $z_n=2 \log n /n$ (and the forthcoming numerical results obtained using this BIC penalty are not surprisingly a disaster). 
\begin{cor}\label{cor1}
Under the conditions of Theorem \ref{theo3}, the bounds \eqref{conv1} and \eqref{conv3} hold, {\it i.e.},
$$
\lim_{\delta \to \infty} \quad\lim_{n\to \infty}\P \Big (\frac {\sqrt m} {n }  \, \big \| \widehat {\bf t} - {\bf t^*}\big \| \geq \delta \Big ) = 0\quad\mbox{and}\quad \sqrt {  m }\, \big \|\widehat {\bf d} -{\bf d^*} \big \|= O_P(1)
$$
\end{cor}
\noindent Then the convergence rates of the estimators obtained in the case where the number of changes is unknown is the same as if the number of changes is known. 
\section{Numerical experiments}\label{simu}
In the sequel we first describe the concrete procedure for applying the new multiple changes estimator, then we present the numerical results of Monte-Carlo experiments.
\subsection{Concrete procedure of estimation}
Several details require to be specified to concretely apply the multiple changes estimator. Indeed, we have done:
\begin{enumerate}
\item The choice of meta-parameters: 1/ as we mainly studied the cases of FARIMA processes for which $\beta=2$, we chose $m=n^{0.65}$; 2/ the number $K_{\max}\geq K^*$ is crucial for the heuristic plot procedure (see below)  and was chosen such as $K_{\max}=2([\log(n)]-1)$, implying $K_{\max}=10,\,12$ and $14$ respectively for $n=500,\, 2000$ and $5000$.
\item As the choice of the sequence $(z_n)$ of the penalty term is not exactly specified but just has to satisfy $\max \big ( z_n \, , \,  \frac {1} {z_n \, \sqrt {  m }} \big ) \limiten 0$. After many numerical simulations, we chose $z_n=2 \,  \sqrt n$ that offers best results among our choices. 
\item The dynamic programming procedure is implemented for allowing a significant decrease of the time consuming. Such procedure is very common in the offline multiple change context and has been described with details in Kay (1998).
\item For improving the procedure of selection of the changes number $K^*$ for not too large samples, we implemented a data-driven procedure so-called ``the heuristic slop procedure''. This procedure was introduced by Arlot and Massart (2009) in the framework of least squares estimation with fixed design, but that can be extended in many statistical fields (see Baudry {\it et al.}, 2012). Applications in the multiple changes detection problem was already successfully done in Baudry {\it et al.} (2012) in an i.i.d. context and also for dependent time series in Bardet {\it et al.} (2012). In a general framework, it consists in computing $-2 \, \log (\widehat  {LIK}(K))$ where $\widehat  {LIK}(K)$ is the maximized likelihood for any $K \in \{0,1,\ldots,K_{\max} \}$. Here $-2 \, \log (\widehat  {LIK}(K))$ is replaced by $\frac 1 n \, \sum_{k=1}^{K+1} n_k\, W_{n}(\widetilde T_k,\widetilde d_k,m)$. Then for $K>K^*$, the decreasing of this contrast with respect to $K$ is almost linear with a slope $s$ (see Figure \ref{Fig1} where the linearity can be observed when $K>K^*=4$), which can be estimated for instance by a least-squares estimator $\widehat s$. Then $\widehat K_H$ is obtained by minimizing the penalized contrast $J_n$ using $\widehat z_n=2 \, \widehat s$, {\it i.e.}
$$
\widehat K_H=\argmin _{0 \leq K \leq K_{\max}} \Big \{ \frac 1 n \, (\widetilde{t}_{k+1}-\widetilde{t}_{k})\,\sum_{k=1}^{K+1} (\, W_{n}(\{\widetilde{t}_k+1,\ldots,\widetilde{t}_{k+1} \},\widetilde d_k,m)+2 \, \widehat s \, K \Big \}. 
$$   
By construction, the procedure is sensitive to the choice of $K_{\max}$ since a least squares regression is realized for the ``largest'' values of $K$ and we preferred to chose the largest reasonable value of $K_{\max}$.
\begin{figure}[ht]
\begin{center}
\includegraphics[width=15 cm,height=10 cm]{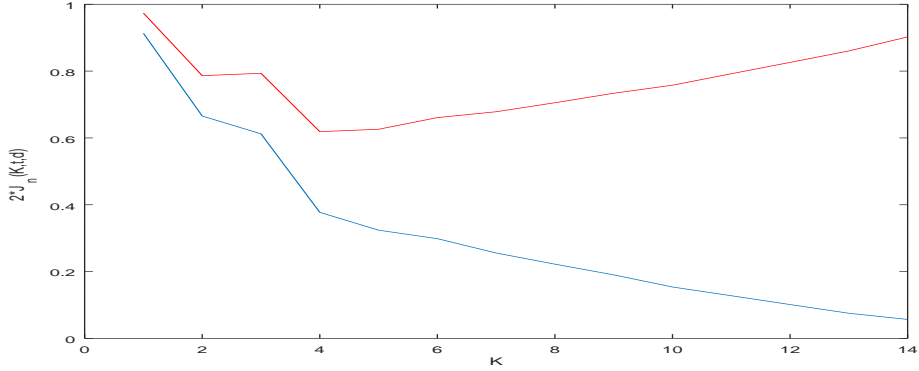}
\caption{\label{Fig1}\textit{For $n=5000$, $K^*=4$ and a FARIMA$(0,d,0)$ process, the graph of $2 \times J_n(K,\,\widehat {\bf t},\, \widehat {\bf d},m)$ (in blue), and the one of $2 \times J_n(K,\,\widehat {\bf t},\, \widehat {\bf d},m)+2 \times \widehat s \times K$ (in red).}} 
\end{center}
\end{figure}
\end{enumerate}
A software was written with {\tt Octave} software (also executable with {\tt Matlab} software) and is available on {\tt http://samm.univ-paris1.fr/IMG/zip/detectchange.zip}.
\subsection{Monte-Carlo experiments in case of known number of changes}
In the sequel we first exhibit the consistency of the multiple breaks estimator when the number of changes is known. Monte-Carlo experiments are realized in the following framework:
\begin{enumerate}
\item Three kinds of processes are considered: a FARIMA$(0,d,0)$ process, a FARIMA$(1,d,1)$ process with a AR coefficient $\psi=-0.7$ and a MA coefficient $\theta=0.3$ (this refers to the familiar representation $(1-\psi\, B)X=(1-B)^{-d} (1+\theta\, B) \varepsilon$ where $B$ is the backward operator) and a linear stationary process called $X^{(d,1)}$ belonging to Class $L(d,1,1)$, since we chose a sequence $(a_k)_{k\in \N}$ satisfying
\begin{eqnarray*}
a_k=(k+1)^{d-1}+ (k+1)^{d-2} \quad \mbox{for all $k \in \N$}.
\end{eqnarray*}
Note that both the FARIMA processes belongs to Class $L(d,2,c_0)$. 
\item For $n=500,\, 2000$ and $5000$, two cases are considered:
\begin{itemize}
\item Zero change, $K^*=0$ and $d_1^*=0.4$, then $d_1^*=0.1$, for obtaining a benchmark of the accuracy of local Whittle estimator of the  long-range dependence parameter;
\item One change, $K^*=1$ and $(d_1^*,d_2^*)=(0.4,0.1)$ and $\tau_1^*=0.5$;
\item Three changes, $K^*=3$ and $(d_1^*,d_2^*,d_3^*,d_4^*)=(0.4,0.1,0.4,0.1)$ and $(\tau_1^*,\tau_2^*,\tau_3^*)=(0.25,0.5,0.75)$. 
\end{itemize}
\item Each case is independently replicated $500$ times and the RMSE, Root-Mean-Square Error, is computed for each estimator of the parameter.  
\end{enumerate}
The results of Monte-Carlo experiments are detailed in Table \ref{Table1}.
\begin{table}
\caption{ \label{Table1} RMSE of the estimators from $500$ independent replications of processes, when the number $K^*$ of changes is known.}
\begin{center}
\begin{tabular}{||l||c||c|c|c||c|c|c||c|c|c||}
& & \multicolumn{3}{c||}{FARIMA$(0,d,0)$} & \multicolumn{3}{c||}{FARIMA$(1,d,1)$ } & \multicolumn{3}{c||}{$X^{(d,1)}$}  \\ \hline
& $n$ & 500 & 2000 & 5000 & 500 & 2000 & 5000 & 500 & 2000 & 5000 \\
\hline\hline 
$K^*=0$ & $\widetilde d_1$~$(d_1=0.4$) & 0.070&0.047 & 0.034 & 0.098  & 0.090 & 0.066 & 0.077  & 0.048 & 0.035   \\   
& $\widetilde d_1$~$(d_1=0.1$)  & 0.075& 0.046 & 0.033 & 0.224 & 0.119 & 0.073 & 0.199 & 0.165 & 0.146 \\  \hline  \hline 
$K^*=1$ & $\widetilde \tau_1$ & 0.202&0.025 & 0.011 & 0.193  & 0.038 & 0.012 & 0.216  & 0.143 & 0.091   \\  
 & $\widetilde  d_1$  & 0.178 & 0.055 & 0.043 & 0.099 & 0.096 & 0.082& 0.189 & 0.130 & 0.092  \\  
& $\widetilde d_2$  & 0.181& 0.063 & 0.043 & 0.317 & 0.188 & 0.128 & 0.258 & 0.162 & 0.130 \\  \hline  \hline 
$K^*=3$ & $\widetilde \tau_1$ & 0.264 & 0.177 &  0.020 & 0.257 & 0.162 & 0.016 &0.197 & 0.175& 0.095 \\  
&$\widetilde \tau_2$ &  0.231 & 0.144 &  0.035 & 0.231 &  0.134 & 0.011 &0.223 & 0.208& 0.141\\  
&$\widetilde \tau_3$ &  0.252 & 0.099 &  0.017 & 0.225  & 0.145 & 0.013 &0.236 & 0.160 & 0.120\\  
 & $\widetilde d_1$  & 0.182 & 0.075 & 0.047 & 0.117 & 0.095 & 0.087 & 0.283 & 0.200 &0.103\\  
& $\widetilde d_2$  & 0.327 & 0.114 & 0.066 & 0.357 & 0.282 & 0.167 & 0.347 & 0.276 &0.167\\ 
 & $\widetilde d_3$  & 0.414 & 0.206 & 0.055 & 0.165 & 0.097 & 0.088 & 0.470 & 0.257 &0.105\\  
& $\widetilde d_4$  & 0.215 & 0.099& 0.061 & 0.365 & 0.293 & 0.196 & 0.308 & 0.206 &0.149\\  
\hline
\end{tabular}
\end{center}
\end{table} 

\subsection{Monte-Carlo experiments in case of unknown number of changes}
In this subsection, we consider the result of the model selection using the penalized contrast for estimating the number of changes $K^*$. We reply exactly the same framework that in the previous subsection and notify the frequencies of the event '$\widehat K=K^*$', for:
\begin{itemize}
\item $\widehat K=\widehat K_n$ obtained directly by minimizing $J_n$ with $z_n=2/\sqrt n$;
\item $\widehat K=\widehat K_{BIC}$ obtained directly by minimizing $J_n$ with $z_n=2\log n/n$, following the usual BIC procedure;
\item $\widehat K=\widehat K_{H}$ obtained from the ``Heuristic Slope'' procedure described previously. 
\end{itemize}
We obtained the results detailed in Table \ref{Table2}:
\begin{table}
\caption{ \label{Table2} Frequencies of recognition of the true number of changes with several criteria from $500$ independent replications of processes.}
\begin{center}
\begin{tabular}{||l||c||c|c|c||c|c|c||c|c|c||}
& & \multicolumn{3}{c||}{FARIMA$(0,d,0)$} & \multicolumn{3}{c||}{FARIMA$(1,d,1)$ } & \multicolumn{3}{c||}{$X^{(d,1)}$}  \\ \hline
& $n$ & 500 & 2000 & 5000 & 500 & 2000 & 5000 & 500 & 2000 & 5000 \\
\hline\hline 
$K^*=1$ & $\widehat K_n$ &0.11 &0.21 & 0.51 & 0.21& 0.45 & 0.67 & 0.05 & 0.05& 0.01 \\ 
 & $\widehat K_{BIC}$  &0  & 0 & 0 & 0 & 0 & 0 &0 & 0 &0 \\  
& $\widehat K_H$  & 0.35 & 0.91 & 0.92 & 0.49 & 0.77 & 0.81 & 0.25  & 0.47 &0.57 \\  \hline  \hline 
$K^*=3$ & $\widehat K_n$ & 0.13& 0.12 & 0.32 &  0.12 & 0.21  &0.52 & 0.16 & 0.07 & 0.02\\  
 & $\widehat K_{BIC}$  & 0 & 0 & 0 & 0 & 0 & 0 & 0 & 0 & 0 \\  
& $\widehat K_H$  & 0.02 & 0.16 & 0.85 & 0.03 & 0.21 & 0.80 &  0.07& 0.16& 0.32\\  
\hline
\end{tabular}
\end{center}
\end{table} 
\subsection{Conclusions of Monte-Carlo experiments}
From Tables \ref{Table1} and \ref{Table2}, we may conclude that:
\begin{enumerate}
\item Even using the local Whittle estimator which is probably the most accurate in this framework, it is easy to verify that if the behaviour of the spectral density in $0$ is not smooth, then even with a trajectory of size 5000, we keep a quadratic risk greater than $0.1$ (see the case $K^*=0$ for a FARIMA$(1,d,1)$ or for the $X^{(d,1)}$ process). We do not have to forget that the parameter $d$ is relative to the long memory behaviour of the process, in a semi-parameteric setting. 
\item If the number of changes is known, the estimators of $\tau_i$ and $d_i$ are consistent but their rates of convergence are slightly impacted by the number of changes: as we could imagine, the largest $K^*$ the largest the RMSE of the estimators. But finally, the case $n=5000$ provides extremely convincing results in FARIMA framework concerning the estimation of $\tau_i$, while the convergence rates for the process $X^{(d,1)}$ are slow (since the asymptotic behavior of the spectral density around $0$ is clearly rougher than in FARIMA framework). 
\item The estimators of number of changes  $\widehat K_n$ and $\widehat K_H$ have a  satisfying behavior, meaning that they seem to converge to $K^*$ when the sample length increases in the FARIMA framework. Once again, the consistencies are slightly better for small $K^*$ than for large $K^*$. The results obtained with the ``Slope Heuristic'' procedure estimator $\widehat K_H$ are almost the most accurate and provides very convincing results for $n=5000$. Note also that the usual BIC penalty is not at all consistent, which can be explained by the use of local Whittle contrast that is not an approximation of the Gaussian likelihood as the usual Whittle contrast is. In case of  process $X^{(d,1)}$, only  $\widehat K_H$ seems to be consistent while  $\widehat K_n$  is not able to detect the number of changes: this is due to the fact that the bandwidth parameter $m$ can not be chosen as $n^{0.65}$ for obtaining consistent estimators of long memory parameters.
\end{enumerate}
Finally we could underline that our detector based on a local Whittle contrast added to a ``slope heuristic'' data-driven penalization provides convincing results when $n=5000$ and not too bad when $n=2000$ (the case $n=500$ gives not significant estimation). 
\section{Proofs} \label{proofs}
\noindent Following the expansion \eqref{Spectral}, we denote in the sequel for $i=1,\ldots,K^*+1$,
\begin{equation}\label{c'}
c^*_{0,i}=\frac {c_i^{*2}}\pi \, B(1-2d_i^*,d_i^*) \, \Gamma(2d_i^*)\, \sin \Big (\frac \pi 2 -\pi d_i^* \Big )\, \big |\lambda \big |^{-2d_i^*}.
\end{equation} 
We first provide the statements and the proofs of two useful lemmas:
\begin{lemma}\label{lem0}
Under the assumptions of Theorem \ref{theo1} and with $S_n(T,d,m)$ defined in \eqref{S0},  for any $i\in \{1,\ldots,K^*+1\}$ and $T \subset T_i^*$, 
%\begin{eqnarray}
%\label{lem1}
%\hspace{-0.5cm}\P \Big (\sup _{d\in (-1/2,1/2)} \max_{n_k\geq N} \frac 1 {n_k } \,\sum_{j\neq j'}^{K^*} \big | R_m(T_{kj},T_{kj'},d) \big | \geq \delta \, \big ( \frac n m \big) ^{2 \min_\ell (d_\ell^*)}\Big ) &\leq& C\, \frac 1 {N^{2  \max_\ell (d_\ell^*)-1}}.
%\end{eqnarray}
%\item 
\begin{multline} \label{lem1}
 \sup _{d \in [0,1/2)}\min\Big ( T\, ,\,  \frac n m \Big )^{-2d_i^*}  \Big |S_n(T,d,m)-\min\Big ( T\, ,\,  \frac n m \Big )^{2d_i^*}  \frac {c^*_{0,i} (2\pi )^{-2d_i^*} }{1+2d-2d_i^*} \Big |  \\
= O_P \Big (  \min \Big (1 \, , \, \frac {n}{m|T|} \Big) ^{1/2} + \Big ( \frac m n \Big ) ^{\beta_i^*} +m^{-2d_i^*} \Big ).
\end{multline}
%\end{enumerate}
\end{lemma}
\begin{proof}
In the sequel, we will use intensively the notation and numerous proofs of Dalla {\it et al.} (2006). However, the results obtained in this paper have to be established again since, we consider $\lambda^{(n)}_{j}=2 \pi \frac j n$ while they considered $\lambda_{j}=2 \pi \frac j {|T|}$. \\
We first define $\displaystyle \eta^*_j=\frac{I_{T}(\lambda^{(n)}_{j})}{c_{0,i}^*\, (\lambda^{(n)}_{j})^{-2d_i^*}}$ and prove:
\begin{equation} \label{vareta}
\E \Big | \frac 1 m \, \sum_{j=1} ^m \big (\eta_j^*- 1 \big ) \Big | \leq C \Big ( \Big ( \frac m n \Big ) ^{\beta_i^*} +\Big (\frac {n}{|T|m}\Big) ^{1/2}\Big )
\end{equation}
where $C>0$ is a constant. For this we will go back to the proof of Proposition 5 in Dalla {\it et al.} (2006). Indeed, with the same notation,  we have:
\begin{eqnarray*}
\E \Big |  \frac 1 m \,  \sum_{j=1} ^m \eta_j^* \Big | & \leq &  \frac 1 m \, \big ( p_{|T|,1}(m)+ p_{|T|,2}(m)+ R_{|T|}(m) \big ) 
\end{eqnarray*}
where $\displaystyle p_{|T|,1}(m)=2\pi \sum_{j=1}^m I_\varepsilon(\lambda^{(n)}_{j})$, $\displaystyle p_{|T|,2}(m)=\sum_{j=1}^m \big (\eta_j- 2\pi I_\varepsilon(\lambda^{(n)}_{j})\big )$ and $\displaystyle R_{|T|}(m)=\sum_{j=1}^m (\eta_j^*-\eta_j)$ with $\displaystyle \eta_j=\frac{I_{T}(\lambda^{(n)}_{j})}{f(\lambda^{(n)}_{j})}$ and $\displaystyle I_\varepsilon(\lambda^{(n)}_{j})=\frac 1 {2\pi |T|} \Big |\sum_{t=1}^{|T|} \varepsilon_t e^{it\lambda^{(n)}_{j}} \Big |^2$. \\
As in Proposition 5 of Dalla {\it et al.} (2006), we can write:
\begin{eqnarray*}
\E |R_{|T|}(m) |& \leq & \sum_{j=1}^m \E \big | \eta_j^*-\eta_j \big | \\
& \leq & C \,m \, \Big ( \frac m n \Big ) ^{\beta_i^*},
\end{eqnarray*}
and therefore
\begin{eqnarray}\label{Rn}
\E \big |R_{|T|}(m)-\E R_{|T|}(m) \big |& \leq & C \,m \, \Big ( \frac m n \Big ) ^{\beta_i^*}.
\end{eqnarray}
Now, following also in Proposition 5 of Dalla {\it et al.} (2006), from Robinson (1995b, Relation (3.17)), adapted with our problem, {\it i.e.} $j \leftrightarrow jT/n$ we have:
\begin{eqnarray}
\nonumber \E \big |\eta_j- 2\pi I_\varepsilon(\lambda^{(n)}_{j})\big |& \leq & C \,\big |\log (1+j|T|/n)\Big |^{1/2} \, \big ( j|T|/n \big ) ^{-1/2} \\
\label{pn2}\Longrightarrow ~\E \big |p_{|T|,2}(m)|  & \leq & C \,\big |\log (1+j|T|/n)\Big |^{1/2} \, \big ( m|T|/n \big ) ^{-1/2}.
\end{eqnarray}
Finally, we have to go back to the proof of (4.9) in Theorem 2 of Robinson (1995b) for bounding $p_{|T|,1}(m)$. Indeed, in this proof and using its notation we have
$$
\E \big | p_{|T|,1}(m)- \E (p_{|T|,1}(m))\big |=\E \Big |\sum_{j=1} 2 \pi \, I_\varepsilon(\lambda^{(n)}_{j}) -1 \Big |\leq \sqrt 2 \, \Big (\Var \Big (\frac m {|T|} \sum_{t \in T} (\varepsilon^2_t -1) \Big )+ \Var \Big ( \sum_{s<t } d_{t-s} \varepsilon_t \varepsilon_s \Big ) \Big )^{1/2}.
$$
But $\displaystyle d_s=\frac 2 {|T|} \, \sum_{j=1} ^m \cos \big ( 2\pi sj/n \big )$ and therefore we easily have $|d_s| \leq 2m /|T|$. Using the usual expression of a sum of cosine functions, we also have $\displaystyle |d_s| \leq \frac 2 {|T|} \, \Big | \frac {\sin (\pi \,sm/n)}{\sin (\pi \,s/n)} \Big | \leq \frac {2n} {\pi s|T|} $. Therefore, using the variance expansion, we deduce that:
$$
\Var \Big (\frac m {|T|} \sum_{t \in T} (\varepsilon^2_t -1) \Big ) \leq C \, \frac {m^2} {|T|},
$$
while the variance of $\sum_{s<t } d_{t-s} \varepsilon_t \varepsilon_s$ is 
\begin{eqnarray*}
O\Big ( |T| \, \sum_{s=1}^{|T|} d_s^2 \Big )&=& O\Big ( |T| \, \sum_{s=1}^{[n/m]} \Big ( \frac {2m}{|T|} \Big )^2 + |T| \, \sum_{s \geq [n/m]} \Big ( \frac {2n}{\pi s |T|} \Big )^2 \Big ) \\
&=& O\Big ( \frac {nm}{|T|} +  \frac {nm}{|T|}  \Big ). 
\end{eqnarray*}
As a consequence we deduce:
\begin{eqnarray}\label{p1}
\E \big | p_{|T|,1}(m)- \E (p_{|T|,1}(m))\big | \leq C\, \Big (\frac {m} {|T|^{1/2}}+\Big (\frac {nm}{|T|}\Big) ^{1/2} \Big ) \leq C \, \Big (\frac {nm}{|T|}\Big) ^{1/2}.
\end{eqnarray}
Finally, using \eqref{Rn}, \eqref{pn2} and \eqref{p1}, we deduce:
$$
\E \Big |  \frac 1 m \,  \sum_{j=1} ^m \eta_j^* \Big | \leq \frac C m \, \Big ( m \Big ( \frac m n \Big ) ^{\beta_i^*} + \log^{1/2} (m|T|/n) \, \big ( m|T|/n \big ) ^{-1/2}+\Big (\frac {nm}{|T|}\Big) ^{1/2} \Big )
$$
and therefore \eqref{vareta} is established. \\
Now a straightforward application of Markov Inequality and Lemma 2 in Dalla {\it et al.} (2006) implies that for any $d \in [0,1/2)$,
\begin{eqnarray} 
\nonumber &&  \Big |\frac 1 m \, \sum_{j=1} ^m \Big (\frac j m \Big ) ^{2d-2d_i^*}\, \eta_j^*- \frac 1 m \, \sum_{j=1} ^m \Big (\frac j m \Big ) ^{2d-2d_i^*}\Big | = O_P \Big ( \Big ( \frac m n \Big ) ^{\beta_i^*} +\Big (\frac {n}{m|T|}\Big) ^{1/2}\Big ) \\
&& \hspace{5mm} \Longrightarrow ~ \Big |\frac 1 m \, \sum_{j=1} ^m \Big (\frac j m \Big ) ^{2d-2d_i^*}\, \eta_j^*- \frac 1 {2d-2d_i^*+1 }\Big | = O_P \Big ( \Big ( \frac m n \Big ) ^{\beta_i^*} +\Big (\frac {n}{m|T|}\Big) ^{1/2}+ m^{2d-2d_i^*-1}\Big ). \label{mieux}
\end{eqnarray}
Since $\displaystyle S_n(T,d,m)=\frac 1 m \, \sum_{j=1}^{m}\big (\frac j m \big )^{2d}I_{T}(\lambda_{j}^{(n)})= \frac {(2\pi )^{2d_i^*}}{c^*_{0,i}} \,  \Big ( \frac n m \Big )^{2d^*_i} \, \frac 1 m \, \sum_{j=1} ^m \Big (\frac j m \Big ) ^{2d-2d_i^*}\, \eta_j^*$, we deduce that for any $N\geq 1$,
\begin{multline} \label{majS}
\sup_{|T|\geq N}\Big |S_n(T,d,m)-\Big ( \frac n m \Big )^{2d_i^*}  \frac {c^*_{0,i} (2\pi )^{-2d_i^*} }{1+2d-2d_i^*} \Big | =\Big ( \frac n m \Big )^{2d_i^*}  \,  O_P \Big ( \Big ( \frac m n \Big ) ^{\beta_i^*} +\Big (\frac {n}{mN}\Big) ^{1/2}+ m^{2d-2d_i^*-1}\Big ).
\end{multline}
For small $N$, for instance such as $N=o( n/m)$, the random right side term is not bounded. However, for any $T \subset T_i^*$, we have $\E \big ( I_{T}(\lambda^{(n)}_{j}) \big ) \leq \sigma_i^2 \big ( 1+2\, C \,   \sum_{k=1}^{|T|}k^{2d_i^*-1} \big ) \leq C \, |T|^{2d_i^*}$. Thus, there exists $C_i>0$ such as for any $\delta>0$, 
\begin{eqnarray}
\label{S2}
\P \Big ( \sup _{d \in [0,1/2)} \Big | S_n(T,d,m)- \frac {c^*_{0,i} (2\pi )^{-2d_i^*} }{1+2d-2d_i^*}  \, |T|^{2d_i^*}\Big |  \geq \delta \Big ) \leq \frac {C_i} \delta  \, |T|^{2d_i^*}.
\end{eqnarray}
Thus we deduce \eqref{lem1} and this achieves the proof of Lemma \ref{lem0}.
\end{proof}
\noindent In the sequel, we define:
\begin{equation}\label{Rnm}
R_n(T,T',d,m)=\frac 1 {2\pi } \, \sum_{t \in T} \sum_{t' \in T'}  X_t \, X_{t'} \, b_n(t'-t,d,m)\quad\mbox{with}\quad b_{n}(k,d,m) = \frac 1 m \, \sum_{j=1}^m \Big ( \frac j m \Big )^{2d} e^{-2 \pi \, i\, \frac{j \,k}{n}}.
\end{equation}
Note that $S_n(T,d,m)$, which is defined in \eqref{S0} can also be written as:
\begin{equation} \label{Sbis}
S_n(T,d,m)=\frac 1 m \, \sum_{j=1}^{m}\big (\frac j m \big )^{2d}I_{T}(\lambda_{j}^{(n)})=\frac 1 {2\pi \, |T|} \, \sum_{s \in T} \sum_{t \in T} X_s \, X_t \, b_n(t-s,d,m).
\end{equation}
The following lemma establish an asymptotic bound for $R_n$ when $T$ and $T'$ are included in distinct stages of the process:
\begin{lemma}\label{lem}
Under the assumptions of Theorem \ref{theo1}, there exists $C>0$ such that for any $j,j' \in \{1,\cdots,K^*+1\}$ where $j\neq j'$, any $T\subset T_j^*$ and $T'\subset T_{j'}^*$, and any $N\in \N^*$, 
\begin{equation}
\label{lem2}
\sup _{d\in [0,1/2)} \max_{\min(|T|,|T'|)\geq N} \!\Big (\min \big (|T|,|T'|, \frac n m \big ) \Big ) ^{- d_j^*-d^*_{j'}} \! \!\!\! \frac 1 {\min(|T|,|T'|)} \, \big | R_n(T,T',d,m) \big |  =  O_P \Big (  \Big (\min \big ( 1 ,  \frac n {mN}\big ) \Big ) ^{1-d_j^*-d^*_{j'}} \Big ).
\end{equation}
%\end{enumerate}
\end{lemma}
\begin{proof}
First, we can bound the covariance $\Cov(X_t,X_t')$ with $t \in T\subset T_{j}$ and $t' \in T'\subset T_{j'}$, where $j\neq j'$. Indeed, assuming $t<t'$,
\begin{eqnarray*}
\Cov(X_t,X_{t'}) &= & \E \Big (\sum_{k=0}^\infty a_k^{(j)} \varepsilon_{t-j}\sum_{k'=0}^\infty a_{k'}^{(j')} \varepsilon_{t'-j'} \Big )\\
&= & \sum_{k=0}^\infty a_k^{(j)}a_{t'-t+k}^{(j')}=\Gamma_{T,T'}(|t'-t|),
\end{eqnarray*}
since $(\varepsilon_i)$ is supposed to be white noise with unit variance. Therefore, since $a_k^{(j)}=c_j^*\, k^{d_j^*-1}+ O\big (k^{d_j^*-1-\beta_j^*} \big) $ and $a_k^{(j')}=c_{j'}^*\, k^{d_{j'}^*-1}+ O\big (k^{d_{j'}^*-1-\beta_{j'}^*} \big) $, there exists $C$ such as
\begin{eqnarray*}
\big | a_k^{(j)}a_{t'-t+k}^{(j')} \big | \leq  C \,k^{d_j^*-1}\, (t'-t+k)^{d_{j'}^*-1}\quad \mbox{for any $k \in \N^*$}.
\end{eqnarray*}
As a consequence, there exist $C'>0$ and $C''>0$ such that for $t'>t$, 
\begin{eqnarray}
\nonumber \big | \Gamma_{T,T'}(|t'-t|)\big | &\leq  & C' \, \sum_{k=1}^\infty k^{d_j^*-1}\, (t'-t+k)^{d_{j'}^*-1} \\
\nonumber &\leq  & \frac {C'}{(t'-t)^{1-d_j^*-d_{j'}^*} }  \times \frac 1 {t'-t} \,  \sum_{k=1}^\infty \Big ( \frac k {t'-t} \Big )^{d_j^*-1}\, \Big (1+\frac k {t'-t} \Big )^{d_{j'}^*-1} \\
\label{tt'} & \leq & \Big (C'' \,  \int_0^\infty \frac 1 {x^{1-d_j^*}} \, \frac 1 {(1+x)^{1-d_{j'}^*}}\, dx \Big ) \,\frac 1 {(t'-t)^{1-d_j^*-d_{j'}^*} }.
\end{eqnarray}
Now, using \eqref{Rnm} and \eqref{tt'}, we have:
\begin{eqnarray*}
\E \big (R_n(T,T',d,m) \big ) &=& \frac 1 {2\pi } \, \sum_{t \in T} \sum_{t' \in T'}  \Cov (X_t, X_{t'} \big ) \, b_n(t'-t,d,m) \\
\Longrightarrow \quad \big | \E \big (R_n(T,T',d,m) \big )  \big |& = & \frac 1 {2\pi } \, \sum_{t \in T} \sum_{t' \in T'} \Gamma_{T,T'}(t'-t)\, b_n(t'-t,d,m).
\end{eqnarray*}
The right side term of the previous equality is only depending on $(t'-t)$. Therefore, using the notations $\delta=-1+\min\{|t-t'|,~(t,t')\in T\times T' \}\geq 0$, $\mu=\min\{|T|,|T'|\}$ and $\nu=\max\{|T|,|T'|\}$, it is possible to detail this term in the following way:
\begin{eqnarray*}
\E \big (R_n(T,T',d,m) \big )& = & \frac 1 {2\pi } \, \Big ( \sum_{k =1} ^\mu k\, \Gamma_{T,T'}(\delta+k) \, b_n(\delta+k,d,m)+ \mu \sum_{k =\mu+1} ^\nu \Gamma_{T,T'}(\delta+k) \, b_n(\delta+k,d,m)\\
&& \hspace{5cm} +\sum_{k =\nu+1} ^{\nu+\mu} (\nu+\mu -k) \,\Gamma_{T,T'}(\delta+k) \, b_n(\delta+k,d,m) \Big ) .
\end{eqnarray*}
But from usual calculations, for any $d \in [-1/2,1/2)$, there exists $C(d)>0$ such as we have
\begin{equation}\label{bn}
|b_{n}(u,d,m)| \leq C(d) \, \min \Big \{ 1 \, , \, \Big ( \frac n m \Big )^{1+2d}|u|^{-1-2d} \Big \}\quad\mbox{for $u \in \Z$}.
\end{equation}
As a consequence, if $\mu+\nu\leq n/m$, we obtain:
\begin{eqnarray}
\nonumber \big | \E \big (R_n(T,T',d,m) \big )  \big |& \leq & C \, \Big ( \sum_{k =1} ^\mu k^{d_j^*+d_{j'}^*}+ 2\mu \sum_{k =\mu+1} ^{\nu+\mu} k^{-1+d_j^*+d_{j'}^*}  \Big )\\
\label{mupetit}& \leq & C \, \mu \, \nu^{d_j^*+d_{j'}^*}.
\end{eqnarray}
And when $\mu\geq n/m$, we can write:
\begin{eqnarray}
\nonumber \sum_{k =1} ^\mu k\, \Gamma_{T,T'}(\delta+k) \, b_n(\delta+k,d,m)& \leq & C\, \sum_{k =1} ^{[n/m]} k^{d_j^*+d_{j'}^*}+C\,  \Big ( \frac n m \Big )^{1+2d} \, \sum_{k =[n/m]} ^\mu k\, \frac {|\delta+k|^{-1-2d}}{|\delta+k|^{1-d_j^*-d_{j'}^*}} \\
\nonumber& \leq & C\, \Big ( \frac n m \Big )^{1+d_j^*+d_{j'}^*}+ C\, \Big ( \frac n m \Big )^{1+2d} \, \sum_{k =[n/m]} ^\mu k^{-1+d_j^*+d_{j'}^*-2d} \\
\nonumber& \leq & C\,  \Big ( \frac n m \Big )^{d_j^*+d_{j'}^*}\, \mu  \, \Big ( \frac n {m \mu} \Big )^{1+\min (0\,,\, 2d-d_j^*-d_{j'})} \, \big (\log (\mu)\big )^{\1_{2d=d_j^*+d_{j'}}}.
\end{eqnarray}

Finally, by performing the same type of calculations several times, we obtain:
\begin{eqnarray}
\nonumber \big | \E \big (R_n(T,T',d,m) \big )  \big |& \leq & C \, \mu \, \Big ( \min \big (\frac n m \, , \, \nu \big )\Big ) ^{d_j^*+d_{j'}^*}\Big ( \min \big ( 1\, ,\, \frac n {m\mu} \big ) \Big )^{1+\min (0\,,\, 2d-d_j^*-d_{j'})} \, \big (\log (\mu)\big )^{\1_{2d=d_j^*+d_{j'}}} \Big ) \\
\label{ER}    \hspace{1cm}\Longrightarrow  \sup_{d \in [0,1/2)} ~\quad&& \hspace{-0.8cm} \big | \E \big (R_n(T,T',d,m) \big )  \big | \leq  C \, \mu \,  \Big ( \min \big (\frac n m \, , \, \nu \big )\Big ) ^{d_j^*+d_{j'}^*} \Big (\min \big ( 1\, ,\, \frac n {m\mu} \big ) \Big )^{1-d_j^*-d_{j'}}\!\!\!\!.
\end{eqnarray}
Now we are going to bound $\Var \big (R_n(T,T',d,m) \big )$. We have:
\begin{eqnarray*}
\Var \big (R_n(T,T',d,m) \big )&=& \frac 1 {4\pi^2 } \, \sum_{t \in T} \sum_{t' \in T'}\sum_{s \in T} \sum_{s' \in T'}  \Cov \big ( X_t X_{t'}\,, \,X_s X_{s'} \big )   \, b_n(t'-t,d,m)\, b_n(s'-s,d,m).
\end{eqnarray*}
Without loss of generality, set $t \leq s<t' \leq s'$. We have: 
\begin{eqnarray*}
\Cov \big ( X_t X_{t'}\,, \,X_s X_{s'} \big )&=& \sum_{k=0}^\infty \sum_{\ell=0}^\infty \sum_{k'=0}^\infty \sum_{\ell'=0}^\infty a^{(j)}_{t-k}a^{(j)}_{s-\ell}a^{(j')}_{t'-k'}a^{(j')}_{s'-\ell'} \, \Cov \big ( \varepsilon_{t-k} \varepsilon_{t'-k'}\,, \,\varepsilon_{s-\ell} \varepsilon_{s'-\ell'} \big ). 
\end{eqnarray*}
Only two cases implies $\Cov \big ( \varepsilon_{t-k} \varepsilon_{t'-k'}\,, \,\varepsilon_{s-\ell} \varepsilon_{s'-\ell'} \big ) \neq 0$ since $(\varepsilon_i)$ is a white noise. For the first one, it is equal to $\mu_4-\sigma^4$ and is obtained when $t-k=t'-k'=s-\ell=s'-\ell'$. For the second one, it is equal to $\sigma^4$ and is obtained when ($t-k=s-\ell)\neq (t'-k'=s'-\ell')$ or $(t-k=s'-\ell')\neq (t'-k'=s-\ell)$. As a consequence,
\begin{eqnarray*}
\Cov \big ( X_t X_{t'}\,, \,X_s X_{s'} \big )&=& (\mu_4-\sigma^4) \,\sum_{k=0}^\infty  a^{(j)}_{k}a^{(j)}_{s-t+k}a^{(j')}_{t'-t+k}a^{(j')}_{s'-t+k} +\sigma^4 \, \sum_{k=0}^\infty\sum_{k'=0, k'\neq k}^\infty a^{(j)}_{k}a^{(j)}_{s-t+k}a^{(j')}_{k'}a^{(j')}_{s'-t'+k'} \\
&& \hspace{7cm}+ \sigma^4 \, \sum_{k=0}^\infty\sum_{\ell=0, \ell\neq k}^\infty a^{(j)}_{k}a^{(j')}_{s'-t+k}a^{(j)}_{\ell}a^{(j')}_{t'-s+ \ell} \\
\Longrightarrow ~\big |\Cov \big ( X_t X_{t'}\,, \,X_s X_{s'} \big ) \big |& \leq & C \,\sum_{k=1}^\infty  \big (k(s-t+k)\big )^{d_j^*-1} \big ( (t'-t+k) (s'-t+k) \big )^{d_{j'}^*-1} \\ && +C \, \Big ( \sum_{k=1}^\infty \big (k(s-t+k)\big )^{d_j^*-1}\Big ) \Big (\sum_{k'=1}^\infty  \big (k'(s'-t'+k')\big )^{d_{j'}^*-1}\Big ) \\ 
&& \hspace{1cm}+ C \, \Big ( \sum_{k=1}^\infty k^{d_j^*-1}(s'-t+k)^{d_{j'}^*-1}\Big ) \Big (\sum_{\ell=1}^\infty  \ell^{d_j^*-1}(t'-s+ \ell)^{d_{j'}^*-1}\Big ).
\end{eqnarray*}
Using the Cauchy-Schwarz Inequality, we have
$$
\sum_{k=1}^\infty  \big (k(s-t+k)\big )^{d_j^*-1} \big ( (t'-t+k) (s'-t+k) \big )^{d_{j'}^*-1} \leq \Big (\sum_{k=1}^\infty  \big (k(s-t+k)\big )^{2d_j^*-2} \Big )^{1/2} \Big ( \sum_{k=1}^\infty  \big ( (t'-t+k) (s'-t+k) \big )^{2d_{j'}^*-2}\Big )^{1/2}
$$
Now we apply the same trick as in \eqref{tt'} and obtain since $s'>t'$,
$$
\sum_{k=1}^\infty  \big (k(s-t+k)\big )^{d_j^*-1} \big ( (t'-t+k) (s'-t+k) \big )^{d_{j'}^*-1} \leq C \, (s-t+1)^{2d_j^*-3/2}  (t'-t+1)^{2d_{j'}^*-3/2},
$$
and more generally,
\begin{eqnarray}
\nonumber \big |\Cov \big ( X_t X_{t'}\,, \,X_s X_{s'} \big ) \big |& \leq & C \Big ( (s-t+1)^{2d_j^*-3/2}  (t'-t+1)^{2d_{j'}^*-3/2} \\ 
\label{covXX}&& +(s-t+1)^{2d_j^*-1}(s'-t'+1)^{2d_{j'}^*-1} + \big ((s'-t)(t'-s)\big )^{d_j^*+d_{j'}^*-1}\Big ).
\end{eqnarray}

\begin{eqnarray}
\nonumber\Var \big (R_n(T,T',d,m) \big )
&\leq & C \, \sum_{t \in T} \sum_{s \in T} \sum_{t' \in T'} \sum_{s' \in T'} \big | \Cov \big ( X_t X_{t'}\,, \,X_s X_{s'} \big ) \,b_{n}(t-t',d,m)b_{n}(s-s',d,m) \big | \\
\label{J1J2J3}&\leq & C \big (J_1+J_2 +J_3 \big ),
\end{eqnarray}
with 
$\displaystyle \left \{\begin{array} {ccl}J_1&=&\displaystyle \sum_{t \in T} \sum_{s \in T} \sum_{t' \in T'} \sum_{s' \in T'} (|t-s|+1)^{2d_j^*-3/2} (|t'-t|+1)^{2d_{j'}^*-3/2}\, \big | b_{n}(t-t',d,m)b_{n}(s-s',d,m) \big | \\ 
J_2&=&\displaystyle \sum_{t \in T} \sum_{s \in T} \sum_{t' \in T'} \sum_{s' \in T'} (1+|t-s|)^{2d_j^*-1} (1+|t'-s'|)^{2d_{j'}^*-1}\, \big | b_{n}(t-t',d,m)b_{n}(s-s',d,m) \big | \\ 
J_3&=&\displaystyle \sum_{t \in T} \sum_{s \in T} \sum_{t' \in T'} \sum_{s' \in T'} \big ((s'-t)(t'-s)\big )^{d_j^*+d_{j'}^*-1}\, \big | b_{n}(t-t',d,m)b_{n}(s-s',d,m) \big |
\end{array} \right . $. \\
As a consequence, we can easily see that $J_1$ is negligible with respect to $J_2$ since $2d-3/2<2d-1$. Concerning $J_2$ we use the same arguments than in Lavielle and Ludena (2000). Then, 
\begin{eqnarray*}
J_2&\leq & C  \sum_{t \in T} \sum_{s \in T}(1+|t-s|)^{2d_j^*-1} \sum_{t' \in T'} \sum_{s' \in T'} (1+|t'-s'|)^{2d_{j'}^*-1} \big |b_{n}(t-t',d,m)b_{n}(s-s',d,m) \big | \\
&\leq & C \,  |T|^{2d_j^*+1} \, \Big ( \sum_{u=0}^{|T|+|T'|} |b_{n}(u,d,m)|^2 +2 \sum _{u=0}^{|T|+|T'|} |b_{n}(u,d,m)| \sum_{v=u+1}^{|T|+|T'|} |b_{n}(v,d,m)| \, |v-u|^{2d_{j'}^*-1}  \Big ).
\end{eqnarray*}
As a consequence, using \eqref{bn},
\begin{multline}\label{b1}
\sum_{u=0}^{|T|+|T'|} |b_{n}(u,d,m)|^2 \leq C \, \Big ( \sum_{u=0}^{\min(|T|+|T'|,n/m)}\hspace{-5mm} 1 +\hspace{-5mm} \sum_{u=\min(|T|+|T'|,n/m)}^{|T|+|T'|}\hspace{-5mm} \big ( n /m \big )^{1+2d}|u|^{-1-2d} \Big ) \\
 \leq C(d) \, \min \big (|T|+|T'|,\frac n m\big ). 
\end{multline} 
Moreover, 
\begin{eqnarray}
\nonumber & & \hspace{-0.5cm} \sum _{u=0}^{|T|+|T'|} |b_{n}(u,d,m)| \sum_{v=u+1}^{|T|+|T'|} |b_{n}(v,d,m)| \, |u-v|^{2d_{j'}^*-1}  \\
\nonumber && \leq C \, \Big \{  \sum_{u=0}^{\min(|T|+|T'|,n/m)} \Big ( \sum_{v=u+1}^{\min(|T|+|T'|,n/m)} \hspace{-5mm}(v-u)^{2d_{j'}^*-1}  + \big ( \frac n m \big )^{1+2d} \hspace{-5mm}\sum_{v={\min(|T|+|T'|,n/m)}}^{|T|+|T'|}\hspace{-5mm} v^{-1-2d}(v-u)^{2d_{j'}^*-1} \Big ) \\
\nonumber && \hspace{6cm}  + \big ( \frac n m \big )^{2+4d} \sum_{u={\min(|T|+|T'|,n/m)}}^{|T|+|T'|}   \sum_{v=u+1}^{|T|+|T'|} (uv)^{-1-2d}(v-u)^{2d_{j'}^*-1}  \Big \} \\
\label{b2}& &\leq C(d) \, \Big ( \min \big (|T|+|T'|,\frac n m\big ) \Big )^{1+2d_{j'}^*}
\end{eqnarray} 
after classical computations. From \eqref{b1} and \eqref{b2}, we obtain:
\begin{equation}\label{b3}
J_2 \leq C (d)\, |T|^{2d_j^*+1} \, \Big ( \min \big (|T|+|T'|,\frac n m\big ) \Big )^{1+2d_{j'}^*}.
\end{equation}
Using the same decomposition of $J_2$ but beginning with $s',t'\in T'$ instead of $s,t\in T$, we can also replace $T$ by $T'$ in the previous bound. As a consequence, we obtain:
\begin{equation}\label{J2fin}
J_2 \leq C (d)\, \min \{ \mu^{2d_j^*+1} \, \Big ( \min \big (\nu,\frac n m\big ) \Big )^{1+2d_{j'}^*} \, , \, \mu^{2d_{j'}^*+1} \, \Big ( \min \big (\nu,\frac n m\big ) \Big )^{1+2d_{j}^*} \Big \}.
\end{equation}
Finally using symmetry reasons we also have $J_3=\big ( \E \big (R_n(T,T',d,m) \big ) \big ) ^2$ and therefore:
\begin{equation}\label{J3fin}
J_3 \leq  C \, \mu^2 \,  \Big ( \min \big (\frac n m \, , \, \nu \big )\Big ) ^{2d_j^*+2d_{j'}^*} \Big (\min \big ( 1\, ,\, \frac n {m\mu} \big ) \Big )^{2-2d_j^*-2d^*_{j'}}.
\end{equation}
As a consequence, using \eqref{J1}, \eqref{J2}, \eqref{J3} and \eqref{J2fin}, \eqref{J3fin}, we obtain that there exists $C>0$ such as:
\begin{equation}\label{VarR}
\sup_{d \in [0,1/2)}  \Var \big (R_n(T,T',d,m) \big )\leq  C \,  \mu ^2 \,\Big ( \min \big (\frac n m \, , \, \nu \big )\Big ) ^{2d_j^*+2d_{j'}^*}  \Big (\min \big ( 1\, ,\, \frac n {m\mu} \big ) \Big )^{2-2d_j^*-2 d^*_{j'}}.
\end{equation}
Therefore, with $\E \big (R^2_n(T,T',d,m) \big )=\Var \big ( R_n(T,T',d,m) \big )+ \E^2 \big ( R_n(T,T',d,m) \big )$, we have for any $N \leq n$, 
\begin{multline}
\sup _{d\in [0,1/2)} \max_{\min(|T|,|T'|)\geq N} \Big ( \min \big (|T|,|T'|\, , \,  \frac n m \big ) \Big ) ^{- d_j^*-d^*_{j'}}   \frac 1 {\min(|T|,|T'|)^2} \, \E \big (R^2_n(T,T',d,m) \big ) \\
\leq C \, \Big (\min \big ( 1\, , \,  \frac n {mN}\big ) \Big ) ^{2-2d_j^*-2 d^*_{j'}}
\label{b5} 
\end{multline}
with $C>0$ that achieves the proof of \eqref{lem2} using Lemma 2.2 and 2.4 in Lavielle and Ludena (2000).
\end{proof}
~\\
\noindent Now the proof of the consistency of $\widehat \tau$ can be established: \\
\begin{proof}[Proof of Theorem \ref{theo1}]
{\it Mutatis mutandis}, we follow here a similar proof than in Lavielle and Ludena (2000). Denote 
\begin{eqnarray} \label{U}
U_n({\bf t},{\bf d},m)=J_n(K^*,{\bf t},{\bf d},m)-J_n(K^*,{\bf t}^*,{\bf d}^*,m),
\end{eqnarray}
where $J_n$ is defined in \eqref{defJ}. Then, using \eqref{Sbis}, we can write that for any ${\bf d}$ and ${\bf t}$, 
$$
U_n({\bf t},{\bf d},m)=\frac 1 n \, \Big [ \sum_{k=1}^{K^*+1} \Big (n_k \, \log \big (S_n(T_{k},d_k,m)\big )- n_k^* \,\log \big ( S_n(T^*_k,d_k^*,m)\big ) \Big ) \Big ] -\frac {\ell_m} n \,  \sum_{k=1}^{K^*+1} 2\, (n_k d_k-n_k^*d^*_k).
$$ 
Now using a decomposition of each $S_n$ on the `true' periods, we can write:
\begin{eqnarray}
\nonumber S_n(T_{k},d_k,m)&=& \sum_{j=1}^{K^*+1} \frac { n_{kj}}{n_k} \, S_n(T_{kj},d_k,m)+ \frac 2 {n_k} \,  \sum_{j=1}^{K^*+1} \sum_{j'=1,~j\neq j}^{K^*}R_n(T_{kj},T_{kj'},d_k,m),
\end{eqnarray}
with $R_n$ defined in \eqref{Rnm}. As a consequence,
\begin{eqnarray} 
\nonumber U_n({\bf d},{\bf t},m)&=& \frac 1 n \, \sum_{k=1}^{K^*+1} \Big [n_k \,  \log \Big (\sum_{j=1}^{K^*+1} \frac { n_{kj}}{n_k} \Big (  S_n(T_{kj},d_k,m)+ \hspace{-5mm}\sum_{j'=1,~j\neq j}^{K^*+1}\frac 2 {n_{kj}} \, R_n(T_{kj},T_{kj'},d_k,m) \Big )\Big )  \\
\nonumber && \hspace{5cm}  -n_k^*  \, \log \big ( S_n(T^*_k,d_k^*,m)\big ) \Big ] +\frac {\ell(m)} n \,  \sum_{k=1}^{K^*+1} 2\, (n_k d_k-n_k^*d^*_k) \\
\nonumber &\geq & \frac 1 n \, \sum_{k=1}^{K^*+1} \sum_{j=1}^{K^*+1}  n_{kj} \, \log \Big (S_n(T_{kj},d_k,m)+  \hspace{-5mm} \sum_{j'=1,~j\neq j}^{K^*+1}\frac 2 {n_{kj}} \, R_n(T_{kj},T_{kj'},d_k,m) \Big )-n_k^*  \, \log \big ( S_n(T^*_k,d_k^*,m)\big ) \Big ] \\
\nonumber && \hspace{5cm} +\frac {\ell(m)} n \,  \sum_{k=1}^{K^*+1} 2\, (n_k d_k-n_k^*d^*_k) \\
\nonumber &\geq & \frac 1 n \, \sum_{k=1}^{K^*+1} \sum_{j=1}^{K^*+1}  n_{kj} \, \Big [ \log \Big (S_n(T_{kj},d_k,m)+  \hspace{-5mm} \sum_{j'=1,~j\neq j}^{K^*+1}\frac 2 {n_{kj}} \, R_n(T_{kj},T_{kj'},d_k,m) \Big )+ 2\,d_k\ell(m) \Big ]  \\
\label{U2} && \hspace{5cm} -\frac 1 n \,  \sum_{k=1}^{K^*+1} n_k^*  \, \Big (\log \big ( S_n(T^*_k,d_k^*,m)\big ) +   2\,d^*_k\ell(m) \Big )  
\end{eqnarray}
using the concavity of $x \mapsto \log(x)$ and with $n_k=\sum_{j=1}^{K^*+1}n_{kj}$. 
Now we are going to use Lemma \ref{lem0} and \ref{lem}. Therefore:
$$
\Big (\frac {n}m \Big )^{-2d_j^*} \Big ( S_n(T_{kj},d_k,m)+  \hspace{-5mm} \sum_{j'=1,~j\neq j}^{K^*+1}\frac 2 {n_{kj}} \, R_n(T_{kj},T_{kj'},d_k,m)\Big ) =\frac {c_{0,j}^*(2\pi)^{-2d_j^*}} {1+2d_k-2d_j^*}+\varepsilon_{kj},
$$
with $\varepsilon_{kj}=O_P(1)$ when $n_{kj}=O(n/m)$ and $\varepsilon_{kj}=o_P(1)$ for $n=o(n_{kj}m)$. As a consequence, 
from \eqref{U2}, Lemma \ref{lem0} and \ref{lem}, we deduce that there exists a random variable $D(m,n)$ such as $D(m,n) \limiteprobanm 0$ satisfying  for any ${\bf t}$ and $\bf d$,
\begin{eqnarray} 
\nonumber  
U_n({\bf t},{\bf d},m) 
\nonumber  & \geq &\sum_{k=1}^{K^*+1}  \sum_{j=1}^{K^*+1} \frac {n_{kj}}n \,  \big ( s(d_j^*,d_k)-s(d_j^*,d_j^*) \big )-|D(m,n)|
\end{eqnarray} 
where for $d\in [0,1/2)$,
\begin{equation} \label{s}
s(d_j^*,d)=2d_j^* \, \log \big ( n/m \big ) + \log \big (c_{0,j}^*(2\pi)^{-2d_j^*} \big ) - \log \big (1+2d-2d_j^* \big )-2d.
\end{equation}
Now, simple computations also imply  
\begin{equation}\label{U4}
U_n({\bf t},{\bf d},m)
\geq   \sum_{k=1}^{K^*+1}  \sum_{j=1}^{K^*+1} \frac {n_{kj}}n \,  \big ( u(d_j^*,d_k)-u(d_j^*,d_j^*) \big ) -|D(m,n)|
\end{equation}
with $u(d_j^*,d)= - \log \big (1+2d-2d_j^* \big )+2d$. Remark that $u(d_j^*,d_k)-u(d_j^*,d_j^*) > 0$ for any $d_j^* \neq d_k$ and of course $u(d_j^*,d_j^*)-u(d_j^*,d_j^*) =0$. Now we could use Lemma 2.3 of Lavielle (1999, p.88), adapted in Lemma 3.3 of Lavielle and Ludena (2000, p.858) and we obtain that there exists $C^*>0$ depending only on ${\bf d}^*$ such as
\begin{eqnarray}\label{lemme23}
\sum_{k=1}^{K^{*}+1}\sum_{j=1}^{K^{*}+1}\frac{n_{kj}}{n}\,\big(u(d_{j}^{*},d_{k})-u(d_{j}^{*},d_{j}^{*})\big)\geq\frac{C}{n}\|{\bf t}-{\bf t^{*}}\|_{\infty},  
\end{eqnarray}
and $\| {\bf t}-{\bf t^*} \|_\infty=\max_{1 \leq k \leq K^*} \big \{ | t_k-t_k^*| \big \}$. \\
Therefore, it is also possible to write that for any $\delta >0$,
\begin{eqnarray}
\nonumber 
\P \big ( \| \widehat \tau - \tau^* \|_\infty > \delta \big ) &\leq & \P \Big (  \inf _{{\bf d} \in [0,1/2)^{K^*+1}} \min_{{\bf t } \in {\cal T}_{K^*}(n \delta)} U_n({\bf t},{\bf d},m) <0  \Big ) \\
\nonumber  & \leq & \P \Big (  \inf _{{\bf d} \in [0,1/2)^{K^*+1}} \min_{{\bf t } \in {\cal T}_{K^*}(n \delta) }   \sum_{k=1}^{K^*+1}  \sum_{j=1}^{K^*+1} \frac {n_{kj}}n \,  \big ( u(d_j^*,d_k)-u(d_j^*,d_j^*) \big )-|D(m,n)| <0  \Big ) \\
\nonumber  & \leq & \P \Big ( \delta -|D(m,n)|<0  \Big ) \limitemn 0,
\end{eqnarray}
since for ${\bf t} \in {\cal T}_{K^*}(n \delta)$ we have $\| {\bf t}-{\bf t^*} \|_\infty \geq \delta \, n$ and for any $k \in\{1,\cdots,K^*\}$. 
This achieves the proof.
\end{proof}

\begin{proof}[Proof of Theorem \ref{theo2}]
Assume  with no loss of generality that $K^*=1$. From Theorem \ref{theo1}, there exists  $ (u_n)_n $
a sequence of real numbers satisfying
  $u_n\sqrt m /n\limiten \infty $, $u_n/n \limiten 0$ ~
  and ~ $ \P \big ( | \widetilde{ t_1}  - {t}_1^* | > u_n \big ) \limiten 0 $. For $\delta > 0$, as we have
$$ \P \Big  ( | \widetilde{ t_1}  - {t}_1^*| > \delta \, \frac {n }{\sqrt m}\Big )
    \leq
    \P\Big  (  \delta \, \frac {n }{\sqrt m}< |\widetilde{ t_1}  - {t}_1^*| \leq  u_n \Big ) + \P\big ( | \widetilde{ t_1}  - {t}_1^* | > u_n \big )
$$
As a consequence, it is sufficient to show that $ \P\big ( \delta \, n /\sqrt m<| \widetilde{ t_1}  - {t}_1^*| \leq u_n  \big ) \limiten 0$. \\
Denote $V_{\delta,n,m}= \{~t \in \Z / ~~  \delta \, n/\sqrt m <  |t_1-t_1^*| \leq u_n
~\}$. Then,
\begin{equation}\label{prob1}
  \P\big(\delta \, \frac {n }{\sqrt m} <| \widetilde{t}_1 - t_1^*| \leq u_n  \big) \leq  \P\Big( \underset{t_1\in V_{\delta,n,m}}{\mbox{min}}\big ({J}_n(K^*,t_1,(\widetilde d_1,\widetilde d_2),m)- {J}_n(K^*,t_1^*,(\widehat d_1^*,\widehat d_2^*),m)\big ) \leq 0 \Big), 
\end{equation}
where $\widehat d_i^*$ are defined in \eqref{Esti*}. \\
Let $t_1 \in V_{\delta,n,m} $ and with no loss of generality chose $t_1> t_1^*$. Then $n_1=t_1$, $n_2=n-t_1$, $n_{11}=t_1^*$, $n_{12}=t_1-t_1^*$, $n_{21}=0$ and $n_{22}=n-t_1$. Then $T_1^*=\{1,\ldots,t_1^*\}$, $T_2^*=\{t_1^*+1+1,\ldots,n\}$, $T_1=\{1,\ldots,t_1\}$, $T_{11}=T_1^*=\{1,\ldots,t_1^*\}$, $T_{12}=\{t_1^*+1,\ldots,t_1\}$, $T_2=\{t_1+1,\ldots,n\}=T_{22}$. \\
On the one hand, using results of Lemma \ref{lem0} and \ref{lem1}, since $t_1/n \limiten \tau_1$ and $(t_1-t_1^*)/n \limiten 0$, we can write $\displaystyle \frac {\frac 1 {t_1-t_1^*} \, R_n(T_{11},T_{12},\widetilde d_1,m)}{S_n(T_1^*,\widetilde d_1,m)} =\Big ( \frac n m \Big )^{1-2d_1^*}O_P\Big (\frac {1}{(t_1-t_1^*)^{1-d_1^*-d_2^*} }\Big ) $. Therefore, using again the concavity of the logarithm function, we have:
\begin{eqnarray*}
J_n(K^*,t_1,(\widetilde d_1,\widetilde d_2),m)&\hspace{-3mm}=&\hspace{-3mm}\frac 1 n \Big \{  t_1 \, \log \Big ( \frac {t_1^*} {t_1}  \, S_n(T_1^*,\widetilde d_1,m)+\frac {t_1-t_1^*} {t_1} \,  S_n(T_{12},\widetilde d_1,m)+ \frac 2 {t_1} \,  R_n(T_{11},T_{12},\widetilde d_1,m) \Big )\\
&&\hspace{1cm}+(n-t_1) \log \big (S_n(T_{22},\widetilde d_2,m)  \big )+2 \, \ell(m) \big ( t_1 \widetilde d_1+ (n-t_1)\widetilde d_2\big ) \Big \} \\
& \hspace{-3mm}\geq &\hspace{-3mm} \frac 1 n \Big \{ t^*_1 \, \log \big (  S_n(T_1^*,\widetilde d_1,m) \big )+ (t_1-t_1^*) \log \big (S_n(T_{12},\widetilde d_1,m) \big )+(n-t_1) \log \big (S_n(T_{22},\widetilde d_2,m)  \big ) \\
&&\hspace{3cm}+2 \, \ell(m) \big ( t_1 \widetilde d_1+ (n-t_1)\widetilde d_2\big )+\Big ( \frac n m \Big )^{1-2d_1^*}O_P\Big (\frac {1}{(t_1-t_1^*)^{1-d_1^*-d_2^*} }\Big )\Big \} \\
& \hspace{-3mm}\geq &\hspace{-3mm} \frac 1 n \Big \{ t^*_1 \, W_n(T_1^*,\widetilde d_1,m) + (t_1-t_1^*) W_n(T_{12},\widetilde d_1,m) +(n-t_1) W_n(T_{22},\widetilde d_2,m) \\
&& \hspace{8cm} +\Big ( \frac n m \Big )^{1-2d_1^*}O_P\Big (\frac {1}{(t_1-t_1^*)^{1-d_1^*-d_2^*} }\Big )\Big \}.
\end{eqnarray*}
On the other hand,  we also have:
\begin{equation*}
J_n(K^*,t_1^*,(\widehat d_1^*,\widehat d_2^*),m)=\frac 1 n \Big \{ t^*_1 \, W_n(T_1^*,\widehat d^*_1,m) +(n-t^*_1) W_n(T_{2}^*,\widehat d_2^*,m) \Big \}.
\end{equation*}
First we remark that from the definition of $\widehat d_1^*$,
\begin{eqnarray*}
W_n(T^*_1,\widehat d_1^*,m) \leq W_n(T^*_1,\widetilde d_1,m).
\end{eqnarray*}
Therefore,
\begin{eqnarray}
\nonumber   &&\hspace{-0.8cm}  
\frac n {t_1-t_1^*} \Big \{ J_n(K^*,t_1,(\widetilde d_1,\widetilde d_2),m)-
J_n(K^*,t_1^*,(\widehat d_1^*,\widehat d_2^*),m) \Big \} \\
\nonumber && \hspace{2cm} \geq \frac 1 {t_1-t_1^*} \, \Big \{(t_1-t_1^*) W_n(T_{12},\widetilde d_1,m) +(n-t_1) W_n(T_{22},\widetilde d_2,m)\\
\label{J1} && \hspace{5cm} - (n-t^*_1) W_n(T_{2}^*,\widehat d_2^*,m) +\Big ( \frac n m \Big )^{1-2d_1^*}O_P\Big (\frac {1}{(t_1-t_1^*)^{1-d_1^*-d_2^*} }\Big )\Big \}.
\end{eqnarray}
Since $t_1 \in V_{\delta,n,m} $, implying $|T_{2}^* |/n \limiteprobamn (1-\tau_1^*)$ and $|T_{22}|/n \limiteprobamn (1-\tau_1^*)$, Lemma \ref{lem0} and more precisely inequality \eqref{majS} can be applied. Then, conditionally to $\widetilde d_1$, $\widetilde d_2$ and $\widehat d_2^*$, we obtain:
\begin{eqnarray*}
W_n(T_{12},\widetilde d_1,m) &=& 2d_2^* \, \log \big ( n/m \big ) + \log \big (c_{0,2}^*(2\pi)^{-2d_2^*} \big ) - \log \big (1+2\widetilde d_1-2d_2^* \big )-2\widetilde d_1 \\
&& \hspace{3cm}+ O_P \Big [\big ( \frac {m } n \big )^{\beta_2^*} + \big (\frac n {m(t_1-t_1^*)}  \big )^{1/2 }+m^{2\widetilde d_1-2d_2^*-1}\Big ] \\
W_n(T_{22},\widetilde d_2,m)&=& 2d_2^* \, \log \big ( n/m \big ) + \log \big (c_{0,2}^*(2\pi)^{-2d_2^*} \big ) - \log \big (1+2\widetilde d_2-2d_2^* \big )-2\widetilde d_2\\
&& \hspace{3cm}+  O_P \Big [ \big ( \frac {m } n \big )^{\beta_2^*} + m^{-1/2 } + m ^{2\widetilde d_2-2d_2^*-1}\Big ] \\ 
W_n(T_{2}^*,\widehat d_2^*,m) &=& 2d_2^* \, \log \big ( n/m \big ) + \log \big (c_{0,2}^*(2\pi)^{-2d_2^*} \big ) - \log \big (1+2\widehat d_2^*-2d_2^* \big )-2\widehat d_2^*\\
&& \hspace{3cm}+  O_P \Big [ \big ( \frac {m } n \big )^{\beta_2^*} + m^{-1/2 }+ m ^{2\widehat d_2^*-2d_2^*-1}\Big ],
\end{eqnarray*}
since $\ell(m)=\frac 1 m \, \sum_{j=1}^m \log(j/m)=-1+ O(m^{-1})$ which is negligible with respect to $O_P  \big ( m^{-1/2 }\big )$. Therefore, \eqref{J1} becomes:
\begin{eqnarray}
\nonumber   &&\hspace{-1.6cm}  
\frac n {t_1-t_1^*} \Big \{ J_n(K^*,t_1,(\widetilde d_1,\widetilde d_2),m)-
J_n(K^*,t_1^*,(\widehat d_1^*,\widehat d_2^*),m) \Big \} \\
\nonumber && \hspace{-0.9cm} \geq \frac 1 {t_1-t_1^*} \, \Big \{-(t_1-t_1^*) \big (\log \big (1+2\widetilde d_1+2d_2^* \big )-2\widetilde d_1 \big ) -(n-t_1)\big ( \log \big (1+2\widetilde d_2-2d_2^* \big )+2\widetilde d_2 \big ) \\
\nonumber && \hspace{0cm}+ \frac n m\, O_P\Big ( \frac {m^{\beta_2^*+1} } {n^{\beta_2^*}} + m^{1/2}+ m ^{2\widetilde d_2-2d_2^*}+ m ^{2\widehat d_2^*-2d_2^*}+ \frac n {t_1-t_1^*}m ^{2\widehat d_1^*-2d_2^*}+\frac {n^{-2d_1^*}}{m^{-2d_1^*}(t_1-t_1^*)^{1-d_1^*-d_2^*} } \Big ) \\
\label{J2} && \hspace{9cm}+ (n-t^*_1) \big (\log \big (1+2\widehat d_2^*-2d_2^* \big )+2\widehat d_2^*\big ) \Big \}.
\end{eqnarray}
$t_1$ is supposed to belong to $ V_{\delta,n,m}$ and therefore $t_1\geq t^*_1+\delta n/\sqrt m$. Moreover, from Dalla {\it et al.} (2006, p. 221), when $m$ is such as $m=o\big (n^{2\beta_2^*/(1+2\beta_2^*)}\big ) $,  then:
\begin{equation}\label{TLCd2}
\widetilde d_2=d_2^* + O_P\big (m^{-1/2} \big )\quad \mbox{and}\quad \widehat d_2^*=d_2^* + O_P\big (m^{-1/2} \big ).
\end{equation}
Then, from \eqref{J2}, we obtain after computations,
\begin{eqnarray}
\nonumber   &&\hspace{-1cm}  
\frac n {t_1-t_1^*} \Big \{ J_n(K^*,t_1,(\widetilde d_1,\widetilde d_2),m)-
J_n(K^*,t_1^*,(\widehat d_1^*,\widehat d_2^*),m) \Big \} \\
\nonumber && \hspace{-0.2cm} \geq 2(d_1^*-d_2^*)-\log \big (1+2(d_1^*-d_2^*) \big ) +\frac n {m(t_1-t^*)} \,   O_P\Big ( \frac {m^{1+\beta_2^*} } {n^{\beta_2^*} }+ \sqrt m \Big )\\
\nonumber && \hspace{-0.2cm} \geq 2(d_1^*-d_2^*)-\log \big (1+2(d_1^*-d_2^*)\big ) + O_P\Big ( \frac 1 {\delta}+ \frac {\sqrt m}{\delta} \big ( \frac {m } n \big )^{\beta_2^*} \Big ).
\end{eqnarray}
As $m =o\big ( n^{2\beta_2^*/(1+2\beta_2^*)} \big)$ then $\sqrt m \big ( \frac {m } n \big )^{\beta_2^*}=o(1)$. As a consequence, we finally obtain: 
\begin{equation}
\label{J3} 
\frac n {t_1-t_1^*} \Big \{ J_n(K^*,t_1,(\widetilde d_1,\widetilde d_2),m)-
J_n(K^*,t_1^*,(\widehat d_1^*,\widehat d_2^*),m) \Big \} \geq 2(d_1^*-d_2^*)-\log \big (1+2(d_1^*-d_2^*) \big )+ O_P\Big ( \frac 1 {\delta} \Big ).
\end{equation}
As $\log(1+x)<x$ for any $x \in (-1,0)\cup (0,1)$, and since $d_1^*-d_2^* \neq 0$, we obtain that 
$$
\lim _{\delta \to \infty} \P \Big ( \frac n {t_1-t_1^*} \Big \{ J_n(K^*,t_1,(\widetilde d_1,\widetilde d_2),m)-
J_n(K^*,t_1^*,(\widehat d_1^*,\widehat d_2^*),m) \Big \} <0 \Big ) =0 
$$
and therefore from \eqref{prob1} we deduce \eqref{conv1} and therefore the proof of Theorem \ref{theo2} is achieved.
\end{proof}
\begin{proof}[Proof of Theorem \ref{theo2bis}]
Using Theorem \ref{theo2}, we can establish that $\widetilde d_i=\widehat d_i^*+  O_P\big (m^{-1/2} \big )$. Indeed, once again without lose of generality, we can consider the case of one change. Using the notation and proof of Theorem \ref{theo2}, if we assume $\widetilde t_1>t_1^*$, knowing $\widetilde t_1-t_1^*\leq C\,\frac n {\sqrt m}$, then $T_2 \subset T_2^*$ and therefore we can again write \eqref{TLCd2} and then $|\widetilde d_2-\widehat d_2^*|= O_P\big (m^{-1/2} \big )$. \\
Concerning $\widetilde d_1$ and with the knowledge that $\widetilde t_1$ is such as $0\leq \widetilde t_1-t_1^*\leq C\,\frac n {\sqrt m}$, we can write that $\widetilde d_1=\argmin_{d \in [0,0.5)} W_n(\{1,\ldots,\widetilde t_1\},d,m)$. But using computations of Theorem \ref{theo2}, we have 
\begin{eqnarray*}
W_n(\{1,\ldots,\widetilde t_1\},d,m)&=&\log \Big ( \frac {t_1^*} {\widetilde t_1}  \, S_n \big (T_1^*,d,m\big)+\frac {\widetilde t_1-t_1^*} {\widetilde t_1} \,  S_n\big(\{t_1^*+1,\ldots,\widetilde t_1\},d,m\big) \\
&& \hspace{2cm}+ \frac 2 {\widetilde t_1} \,  R_n\big(\{1,\ldots,t^*_1\},\{t_1^*+1,\ldots,\widetilde t_1\},d,m\big)\Big ) +2 d\, \ell(m) \\
&=&\log \Big ( S_n \big (T_1^*,d,m\big) \Big )+ D_{m,n,d} \, \Big ( \frac {\widetilde t_1-t_1^*}{t_1^*}\Big ) +2 d\, \ell(m)+\log(t_1^*/\widetilde t_1) \\
&=& W_n(T_1^*,d,m)+D_{m,n,d} \, \Big ( \frac {\widetilde t_1-t_1^*}{n}\Big ),
\end{eqnarray*}
where $\sup_{d\in [0,1/2)} |D_{m,n,d}| =O_P(1)$  using Lemmas \ref{lem0} and \ref{lem1} and because we have $t_1^*=[n\tau_1^*]$. Now, since $\widehat d_1^*=\argmin_{d \in [0,0.5)} W_n(T_1^*,d,m)$ and $d \in [0,1/2) \mapsto W_n(T,d,m)$ is a ${\cal C}^1([0,1/2))$ function, we deduce that $\widetilde d_1=\widehat d_1^*+ \frac 1 n \, O_P\big (|\widetilde t_1-t_1^*| \big )=\widehat d_1^*+ O_P\big (m^{-1/2} \big )$. This achieves the proof of Theorem \ref{theo2bis}.
\end{proof}

\begin{proof}[Proof of Theorem \ref{theo3}]
Obiously, the proof is established if for any $K \in \big \{0,\ldots, K^*-1,K^*+1,\ldots,K_{\max} \big \}$ the following consistency holds:
\begin{equation}\label{convK}
\P \Big ( J_n(K,{\bf t}, {\bf d},m)-J_n(K^*, {\bf t^{*}},{\bf d^{*}},m) <0\Big ) \limiten 0,
\end{equation}
for any ${\bf t}$ and ${\bf d}$, with $J_n$ defined as in \eqref{defJ}. Indeed, as $J_n(K^*,\widehat {\bf t^{*}},\widehat {\bf d^{*}},m) \leq J_n(K^*, {\bf t^{*}},{\bf d^{*}},m)$ by definition, \eqref{convK} is also satisfied by replacing $J_n(K^*, {\bf t^{*}},{\bf d^{*}},m)$ by $J_n(K^*,\widehat {\bf t^{*}},\widehat {\bf d^{*}},m)$. 
We decompose the proof in two parts, $K<K^*$ and $K>K^*$. \\
~\\
{\bf Assume {\bf $K<K^*$}}. Then, for any ${\bf t}$ and ${\bf d}$, and using \eqref{U2},
\begin{eqnarray*}
J_n(K,{\bf t}, {\bf d},m)-J_n(K^*, {\bf t^{*}},{\bf d^{*}},m)&& \\
&& \hspace{-5cm}=\frac 1 n \, \sum_{k=1}^{K+1} n_k \,  \log \Big (\sum_{j=1}^{K^*+1} \frac { n_{kj}}{n_k} \, S_n(T_{kj},d_k,m)+  \frac 2 {n_k} \,\sum_{j=1}^{K^*+1} \sum_{j'=1,~j\neq j}^{K^*+1}R_n(T_{kj},T_{kj'},d_k,m) \Big ) \\
&& \hspace{-4cm} -\frac 1 n \, \sum_{j=1}^{K+1} n_j^*  \, \log \big ( S_n(T^*_j,d_j^*,m)\big ) +2 \, \frac {\ell(m)} n \, \Big ( \sum_{k=1}^{K+1} n_k d_k-\sum_{j=1}^{K^*+1} n_j^*d^*_j \Big )+(K-K^*)z_n\\
&& \hspace{-5cm}\geq \sum_{k=1}^{K+1}  \sum_{j=1}^{K^*+1} \frac {n_{kj}}n \,  \big ( s(d_j^*,d_k)-s(d_j^*,d_j^*) \big )-|D(m,n)|+(K-K^*)z_n \\
&& \hspace{-5cm}\geq \sum_{k=1}^{K+1}  \sum_{j=1}^{K^*+1} \frac {n_{kj}}n \,  \big ( u(d_j^*,d_k)-u(d_j^*,d_j^*) \big ) -|D(m,n)|+(K-K^*)z_n
\end{eqnarray*}
since $\sum_{j=1}^{K^*+1} n_{kj}=n_k$ and  $\sum_{k=1}^{K+1} n_{kj}=n_j^*$ and using \eqref{s} with $D(m,n) \limiteprobanm 0$ and $u(d_j^*,d)= - \log \big (1+2d-2d_j^* \big )+2d \geq 2d_j^*$.\\
Now, we use again Lemma 2.3 of Lavielle (1999, p. 88). This Lemma was obtained when $K=K^*$ and we obtain that there exist $C_d>0$ such as 
$$
\sup_{{\bf d}\in ,{\bf t}\in }\sum_{k=1}^{K^*+1}  \sum_{j=1}^{K^*+1} \frac {n_{kj}}n \,  \big ( u(d_j^*,d_k)-u(d_j^*,d_j^*) \big )\geq  C_d \, \frac 1 n \, \| {\bf t}-{\bf t^*}\|_\infty
$$
where $\| {\bf t}-{\bf t^*}\|_\infty =\max_{1 \leq j \leq K^*} |t_j-t_j^*|$. However this result is still valide when  $K^*$ is replaced by $K<K^*$ in the first sum, since it is sufficient to add $K^*-K$ fictive times and consider $t_{K+1}=t_{K+2}=\cdots =t_{K^*}=t_K$ (and therefore $n_{kj}=0$ for $k=K+2,\ldots,K^*+1$. Therefore we obtain:
\begin{equation} \label{InegJ}
J_n(K,{\bf t}, {\bf d},m)-J_n(K^*, {\bf t^{*}},{\bf d^{*}},m) \geq \frac 1 3 \, \min_{1\leq i \leq K^*} |\tau_{i+1}^*-\tau_i^*|  -|D(m,n)|+(K-K^*)z_n
\end{equation}
since $K<K^*$ and therefore $\| {\bf t}-{\bf t^*}\|_\infty \geq \frac 1 2 \, \min_{1\leq i \leq K^*} |t_{i+1}^*-t_i^*| \geq \frac n 3 \, \min_{1\leq i \leq K^*} |\tau_{i+1}^*-\tau_i^*| $ when $n$ is large enough. Therefore, if $z_n \limiten 0$ then \eqref{convK} is satisfied and therefore $\P \big ( \widehat K<K^*) \limiten 0$.  
~\\
\noindent {\bf Assume {\bf $K^*<K\leq K_{\max}$}}. With  ${\bf  \widehat{t}} = (\widehat{t}_{1}, \ldots,
\widehat{t}_{K}) $, there exists some subset
$ \{ k_j ~, 1
 \leq j \leq K^* \} $ of $ \{ 1, \ldots,  K \}  $ such that for any $ j = 1, \ldots,  K^* $, $\big |\frac {\widehat t_{k_j}} n-\tau^*_j \big | =O_P\Big (\frac {1}  {\sqrt m} \Big )$. To see this, consider the $({\widehat{t}_{k_j}})$ as the closest times among $(\widehat{t}_{1}, \ldots,
\widehat{t}_{K}) $ to the $(t_1^*,\ldots,t_{K^*}^*)$. The other $K-K^*$ change dates $\widehat t_i$ could be consider exactly as additional ``false'' changes (since the parameters $d$ do not change at these times) and therefore the $\widehat t_{k_j}$ minimize $ J_n(K,{\bf t}, {\bf d},m)$ conditionally to those $\widehat t_i$ with $i \notin \{k_1,\ldots,k_{K^*}$ as if the number of changes is known and is $K^*$. And therefore Theorem \ref{theo2} holds for those $\widehat t_{k_j}$. \\
Then using the previous expansions detailed in the previous proofs, we obtain 
\begin{eqnarray*}
J_n(K,{\widehat {\bf t}}, {\widehat {\bf d}},m) -J_n(K^*, {\bf t^{*}},{\bf d^{*}},m)  \\
&& \hspace{-5cm}=\frac 1 n \, \sum_{j=1}^{K^*+1} \Big ( \sum_{k=k_j+1}^{k_{j+1}} \widehat n_k \,  \log \big (S_n(\widehat T_{k},\widehat d_k,m)\big ) - n_j^*  \, \log \big ( S_n(T^*_j,d_j^*,m)\big )\Big ) \\
&& \hspace{-0cm}+2 \, \frac {\ell(m)} n \, \Big ( \sum_{i=1}^{K+1} \widehat n_i \widehat d_i-\sum_{j=1}^{K^*+1} n_j^*d^*_j \Big )+(K-K^*)z_n\\
&& \hspace{-5cm}\geq \frac 1 n \, \sum_{j=1}^{K^*+1} \Big ( \sum_{k=k_j+1}^{k_{j+1}} \widehat n_k \,  s(d_j^*,\widehat d_k)- n_j^*  \, s(d_j^*,d_j^*)\Big )-|D(m,n)| \\
&& \hspace{-0cm}+2 \, \frac {\ell(m)} n \, \Big ( \sum_{i=1}^{K+1} \widehat n_i \widehat d_i-\sum_{j=1}^{K^*+1} n_j^*d^*_j \Big ) +(K-K^*)z_n \\
&& \hspace{-5cm}\geq \frac 1 n \, \sum_{j=1}^{K^*+1} \Big ( \sum_{k=k_j+1}^{k_{j+1}} \widehat n_k \,  s(d_j^*,\widehat d_k)- n_j^*  \, s(d_j^*,d_j^*)\Big )-|D(m,n)| \\
&& \hspace{-0cm}+2 \, \frac {\ell(m)} n \, \Big ( \sum_{i=1}^{K+1} \widehat n_i \widehat d_i-\sum_{j=1}^{K^*+1} n_j^*d^*_j \Big ) +(K-K^*)z_n
\end{eqnarray*}
with $s$ defined in \eqref{s}. Now, since $\widehat T_k \subset \big \{\widehat t_{k_j+1},\ldots,\widehat t_{k_{j+1}} \big \}$, we have from Theorem \ref{theo3}, $\widehat d_k=d_j^*+ O \Big (\frac {1} {\sqrt {  m }}  \Big )$. As a consequence, for $k=k_j+1, \ldots,k_{j+1}$ then $s(d_j^*,\widehat d_k)=s(d_j^*,d_j^*)+ O_P \Big (\frac {1} {\sqrt {  m }}  \Big )$. Then,
\begin{eqnarray*}
J_n(K,{\widehat {\bf t}}, {\widehat {\bf d}},m) -J_n(K^*, {\bf t^{*}},{\bf d^{*}},m)  \\
&& \hspace{-5cm}\geq \frac 1 n \, \sum_{j=1}^{K^*+1} \Big ( s(d_j^*,d_j^*) + 2 \, d_j^*\frac {\ell(m)} n \Big )\, \Big ( \sum_{k=k_j+1}^{k_{j+1}} \widehat n_k - n_j^*  \Big )-|D(m,n)|-|E(m,n)|  +(K-K^*)z_n\\
&& \hspace{-5cm}\geq  -|D(m,n)|-|E'(m,n)|  +(K-K^*)z_n,
\end{eqnarray*}
with $D(m,n)=O_P \Big (\frac {1} {\sqrt {  m }}  \Big )$ under condition $m =o\big ( n^{2\underline \beta^*/(1+2\underline \beta^*} \big)$ from the proof of Theorem \ref{theo2}, $E(m,n)=O_P \Big (\frac {1} {\sqrt {  m }}  \Big )$ and therefore $E'(m,n)=O_P \Big (\frac {1} {\sqrt {  m }}  \Big )$ since $\Big |\sum_{k=k_j+1}^{k_{j+1}} \widehat n_k - n_j^* \Big | =O_P\Big (\frac {n }  {\sqrt m} \Big )$.\\
As a consequence if $(z_n)$ is such that $z_n \,\sqrt {  m } \limiten \infty$ then for any $K> K^*$,
$$
\P \Big (J_n(K,{\widehat {\bf t}}, {\widehat {\bf d}},m) -J_n(K^*, {\bf t^{*}},{\bf d^{*}},m)<0 \Big )\limiten 0.
$$
This achieves the proof.
\end{proof}

\begin{proof}[Proof of Corollary \ref{cor1}] The results are easily obtained by considering conditional probability with respect to the event $\widehat K =K^*$.
\end{proof}

\section*{Aknowledgement} 
The authors thank the Associate Editor and the referees for their fruitful corrections, comments and suggestions, which notably improved the quality of the paper. 

\end{document}